\documentclass[12pt]{amsart}
\usepackage{amssymb,amsfonts,amsmath,amscd}
\usepackage[all]{xy}

\newtheorem{Theorem}{\bf Theorem}[section]
\newtheorem{Proposition}[Theorem]{\bf Proposition}
\newtheorem{Lemma}[Theorem]{\bf Lemma}
\newtheorem{Corollary}[Theorem]{\bf Corollary}

\newtheorem{Remark}[Theorem]{\bf Remark}
\newtheorem{Example}[Theorem]{\bf Example}
\newtheorem{Definition}[Theorem]{\bf Definition}

\newcommand{\fa}{\mbox{$\mathfrak{a}$}}
\newcommand{\fb}{\mbox{$\mathfrak{b}$}}
\newcommand{\fm}{\mbox{$\mathfrak{m}$}}
\newcommand{\fn}{\mbox{$\mathfrak{n}$}}
\newcommand{\fp}{\mbox{$\mathfrak{p}$}}
\newcommand{\fq}{\mbox{$\mathfrak{q}$}}
\newcommand{\bn}{\mbox{$\mathbb{N}$}}
\newcommand{\bz}{\mbox{$\mathbb{Z}$}}
\newcommand{\bq}{\mbox{$\mathbb{Q}$}}

\newcommand{\bc}{\mbox{$\mathbb{C}$}}
\newcommand{\ba}{\mbox{$\mathbb{A}$}}
\newcommand{\cm}{\mbox{$\mathcal{M}$}}

\begin{document}

\title{Ideals of Herzog-Northcott type}

\author{{\sc Liam O'Carroll and Francesc Planas-Vilanova}}

\date{\today}


\begin{abstract}
This paper takes a new look at ideals generated by $2\times 2$ minors
of $2\times 3$ matrices whose entries are powers of three elements not
necessarily forming a regular sequence. A special case of this are the
ideals determining monomial curves in three dimensional space, which
were already studied by Herzog. In the broader context studied here,
these ideals are identified as Northcott ideals in the sense of
Vasconcelos, and so their liaison properties are displayed. It is
shown that they are set-theoretically complete intersections,
revisiting the work of Bresinsky and of Valla. Even when the three
elements are taken to be variables in a polynomial ring in three
variables over a field, this point of view gives a larger class of
ideals than just the defining ideals of monomial curves. We then
characterize when the ideals in this larger class are prime, we show
that they are usually radical and, using the theory of multiplicities,
we give upper bounds on the number of their minimal prime ideals, one
of these primes being a uniquely determined prime ideal of definition
of a monomial curve. Finally, we provide examples of
characteristic-dependent minimal prime and primary structures for these
ideals.
\end{abstract}


\maketitle

\section{Introduction}\label{introduction}

Let $A$ be a commutative Noetherian ring (with identity) and
$x_{1},x_{2},x_{3}$ a sequence of elements of $A$ generating a proper
ideal of height 3. Set $\bn_{0}=\bn\cup\{0\}$ and take
$a=(a_{1},a_{2},a_{3})\in \bn_{0}^{3}$ and
$b=(b_{1},b_{2},b_{3})\in\bn_{0}^{3}$. Let $c=a+b$,
$c=(c_{1},c_{2},c_{3})$. Let $\cm$ be the matrix
\begin{eqnarray*}
\cm=\left(\begin{array}{ccc}x_{1}^{a_{1}}&x_{2}^{a_{2}}&x_{3}^{a_{3}}
  \\ x_{2}^{b_{2}}&x_{3}^{b_{3}}&x_{1}^{b_{1}}\end{array}\right), 
\end{eqnarray*}
and $v_{1}=x_{1}^{c_{1}}-x_{2}^{b_{2}}x_{3}^{a_{3}}$,
$v_{2}=x_{2}^{c_{2}}-x_{1}^{a_{1}}x_{3}^{b_{3}}$ and
$D=x_{3}^{c_{3}}-x_{1}^{b_{1}}x_{2}^{a_{2}}$, the $2\times 2$ minors
of $\cm$ up to a change of sign. Consider
$I_{2}(\cm)=(v_{1},v_{2},D)$, the determinantal ideal generated by the
$2\times 2$ minors of $\cm$. Note that
$x_{2}^{b_{2}}D=-x_{3}^{b_{3}}v_{1}-x_{1}^{b_{1}}v_{2}$, so that, if
$b_{2}=0$, then $I_{2}(\cm)=(v_{1}, v_{2})$.

Our motivation to consider these ideals comes from the following
well-known result of Herzog in \cite{herzog} (see also \cite[pages
  138-139]{kunz}). Take the irreducible affine space curve of
$\ba^{3}_{k}=k^{3}$, $k$ a field, given by the parametrization
$x_{1}=t^{n_{1}}$, $x_{2}=t^{n_{2}}$, $x_{3}=t^{n_{3}}$, where
$n=(n_{1},n_{2},n_{3})\in\bn^{3}$ ($n_{i}>0$), with ${\rm
  gcd}(n_{1},n_{2},n_{3})=1$. Let $\fp_{n}$ be the vanishing ideal of
this curve, i.e., the height 2 prime ideal of the polynomial ring in
three variables $A=k[x_{1},x_{2},x_{3}]$ defined as the kernel of the
natural morphism $\varphi:k[x_{1},x_{2},x_{3}]\rightarrow k[t]$,
$\varphi (x_{i})=t^{n_{i}}$. We will call $\fp_{n}$ the {\em Herzog
  ideal associated to} $n=(n_{1},n_{2},n_{3})$ (see
Section~\ref{prime} for more details concerning this definition; note
that Huneke in \cite{huneke1} used the term Herzog ideals for a
different class of ideals).  Herzog proved that $\fp_{n}$ is either a
complete intersection or an almost complete intersection ideal (in the
sense of \cite{hmv}). Concretely, with a suitable numbering of the
variables, $\fp_{n}$ has a set of generators of one of the two
following types:
\begin{enumerate}
\item[(ci):] $v_{1}=x_{1}^{c_{1}}-x_{3}^{a_{3}}$,
  $v_{2}=x_{2}^{a_{2}}-x_{1}^{a_{1}}x_{3}^{b_{3}}$,
  $c_{1},a_{2},a_{3}\in\bn$, $a_{1}b_{3}\neq 0$ (here $b_{2}=0$);
\item[(aci):] $v_{1}=x_{1}^{c_{1}}-x_{2}^{b_{2}}x_{3}^{a_{3}}$,
  $v_{2}=x_{2}^{c_{2}}-x_{1}^{a_{1}}x_{3}^{b_{3}}$ and
  $D=x_{3}^{c_{3}}-x_{1}^{b_{1}}x_{2}^{a_{2}}$, with $a,b\in\bn^{3}$.
\end{enumerate}
In other words, a Herzog ideal can always be seen as an ideal of the
type $I_{2}(\cm)=(v_{1},v_{2},D)$, for some appropriate
$a,b\in\bn_{0}^{3}$. However, even in the case of a polynomial ring,
an ideal of the form $I_{2}(\cm)=(v_{1},v_{2},D)$ is not always a
Herzog ideal because (as we will show) it might not be a prime ideal
(see Theorem~\ref{theorem-prime}), whereas a Herzog ideal is prime by
definition.

Herzog also proved that $\fp_{n}$ is always a set-theoretic complete
intersection. Subsequently Bresinsky and Valla gave constructive
proofs that an ideal of the form $I_{2}(\cm)$ is a set-theoretic
complete intersection (\cite{bresinsky} and
\cite[Theorem~3.1]{valla}), the first in the polynomial case
$A=k[x_{1},x_{2},x_{3}]$, the latter in our setting and using a
general result on determinantal ideals proved by Eagon and Northcott
(\cite[Theorem~3]{en}).

We recently made use of Herzog ideals in order to produce a negative
answer to a long-standing question about the uniform Artin-Rees
property on the prime spectrum of an excellent ring. Concretely, for
$s\in\bn$, $s\geq 4$, and $n_{1}(s)= s^{2}-3s+1$,
$n_{2}(s)=s^{2}-3s+3$ and $n_{3}(s)=s^{2}-s+1$,
$n(s)=(n_{1}(s),n_{2}(s), n_{3}(s))\in\bn^{3}$, the one parameter
family of (non complete intersection) Herzog ideals $\fp_{n(s)}$
satisfies the relation that, for all $s\geq 4$, $\fp_{n(s)}^{s}\cap
x_{3}A\varsupsetneq \fp_{n(s)}(\fp_{n(s)}^{s-1}\cap x_{3}A)$ (see
\cite{op}).

On the other hand, in \cite{northcott}, Northcott considered the
following situation: let $u=u_{1},\ldots ,u_{r}$ and $v=v_{1},\ldots
,v_{r}$ be two sets of $r$ elements of a Noetherian ring $A$, connected
by the relations:
\begin{eqnarray*}
\left\{ \begin{array}{ccc}
v_{1}&=&a_{1,1}u_{1}+a_{1,2}u_{2}+\ldots +a_{1,r}u_{r},\\
v_{2}&=&a_{2,1}u_{1}+a_{2,2}u_{2}+\ldots +a_{2,r}u_{r},\\
&&\ldots \\
v_{r}&=&a_{r,1}u_{1}+a_{r,2}u_{2}+\ldots +a_{r,r}u_{r},
\end{array}\right.
\end{eqnarray*}
with $a_{i,j}\in A$. Let $D$ stand for the determinant of the $r\times
r$ matrix $\Phi=(a_{i,j})$. Northcott proved that if $(v_{1},\ldots ,
v_{r})$ has grade $r$ and $(v_{1},\ldots ,v_{r},D)$ is proper, then
the projective dimension of $A/(v_{1},\ldots ,v_{r},D)$ is $r$, and
$(v_{1},\ldots ,v_{r},D)$ and all its associated prime ideals have
grade $r$ (\cite[Theorem~2]{northcott}). Subsequently, Vasconcelos
called such an ideal $(v_{1},\ldots ,v_{r},D)$ the {\em Northcott
  ideal associated to $\Phi$ and $u$} (see e.g. \cite[page
  100]{vasconcelos}).

Take now $r=2$, $u_{1}=x_{1}^{a_{1}}$, $u_{2}=-x_{2}^{b_{2}}$ and
$v_{1}=x_{1}^{c_{1}}-x_{2}^{b_{2}}x_{3}^{a_{3}}$,
$v_{2}=x_{2}^{c_{2}}-x_{1}^{a_{1}}x_{3}^{b_{3}}$ the aforementioned
first two generators of $I_{2}(\cm)=(v_{1},v_{2},D)$. Let $\Phi$
be the $2\times 2$ matrix defined by

\begin{eqnarray*}
\Phi=\left( \begin{array}{rr}
x_{1}^{b_{1}}&x_{3}^{a_{3}}\\-x_{3}^{b_{3}}&-x_{2}^{a_{2}}
\end{array}\right),
\end{eqnarray*}
whose determinant, note, is just
$D=x_{3}^{c_{3}}-x_{1}^{b_{1}}x_{2}^{a_{2}}$, the third generator of
$I_{2}(\cm)$. We clearly have $\Phi\cdot[u]^{\top}=[v]^{\top}$. In
other words, the ideal $I_{2}(\cm)=(v_{1},v_{2},D)$ (in particular, a
Herzog ideal) can be viewed as a Northcott ideal whenever the ideal
$(v_{1},v_{2})$ has grade $2$. This fact, though simple, was extremely
useful in proving the main result in \cite{op}. As a consequence, it
awakened our interest in this family of ideals. In this paper, we
intend to study their general properties, though we will restrict
ourselves just to the case $a,b\in\bn^{3}$, i.e., $a_{i},b_{j}>0$. For
ease of reference, we state the following definition.

\begin{Definition}{\rm
Let $A$ be a commutative Noetherian ring and $x=x_{1},x_{2},x_{3}$ a
sequence of elements of $A$ generating an ideal of height 3. Let
$a=(a_{1},a_{2},a_{3})\in\bn^{3}$ and
$b=(b_{1},b_{2},b_{3})\in\bn^{3}$ ($a_{i},b_{j}>0$) and set $c=a+b$,
$c=(c_{1},c_{2},c_{3})$. Let $\cm$ be the matrix
\begin{eqnarray*}
\cm=\left(\begin{array}{ccc}x_{1}^{a_{1}}&x_{2}^{a_{2}}&x_{3}^{a_{3}}
  \\ x_{2}^{b_{2}}&x_{3}^{b_{3}}&x_{1}^{b_{1}}\end{array}\right), 
\end{eqnarray*}
and $v_{1}=x_{1}^{c_{1}}-x_{2}^{b_{2}}x_{3}^{a_{3}}$,
$v_{2}=x_{2}^{c_{2}}-x_{1}^{a_{1}}x_{3}^{b_{3}}$ and
$D=x_{3}^{c_{3}}-x_{1}^{b_{1}}x_{2}^{a_{2}}$, the $2\times 2$ minors
of $\cm$ up to a change of sign. The ideal
$I=I_{2}(\cm)=(v_{1},v_{2},D)$ will be called the {\em
  Herzog-Northcott} (HN, for short) {\em ideal associated to} $\cm$ .
}\end{Definition}

The paper is organized as follows. In Section~\ref{basic} we start
with a few preliminary results. In Section~\ref{stci}, following
Bresinsky's ideas in \cite{bresinsky}, we prove, in our general
setting, that HN ideals are set-theoretically complete intersections,
thus recovering the result of Valla. Concretely, we prove that the
element $g_{B}$ considered by Bresinsky verifies ${\rm rad}(I)={\rm
  rad}(v_{1},g_{B})$ in complete generality (see
Theorem~\ref{gb}). Comparing the element $g_{B}$ given by Bresinsky
and the one $g_{V}$ given by Valla, we show that they are equal modulo
$(v_{1})$ under mild hypotheses and are in general closely
related. Section~\ref{secgradev=2} is mainly devoted to studying the
condition ${\rm grade}(v_{1},v_{2})=2$. In Sections~\ref{licci} and
\ref{aci}, and supposing that $x=x_{1},x_{2},x_{3}$ is a regular
sequence and that $(v)$ has grade 2, we prove that HN ideals are
geometrically linked to a complete intersection and that they are
almost complete intersections (in the sense of \cite{hmv}). From
Section~\ref{prime} to the end of the paper, we restrict ourselves to
the case of the polynomial ring $A=k[x_{1},x_{2},x_{3}]$, $k$ a field
and $x=x_{1},x_{2},x_{3}$ three variables over $k$. Then, in
Section~\ref{prime}, we characterize when an HN ideal is a prime
ideal. To do this, we first consider an integer vector
$m(I)=(m_{1},m_{2},m_{3})\in\mathbb{N}^{3}$ associated to $I$ and the
corresponding Herzog (prime) ideal $\fp_{m(I)}$. Then we prove that
$I$ being prime, ${\rm gcd}(m(I))=1$ and $I=\fp_{m(I)}$ are three
equivalent conditions (see
Theorem~\ref{theorem-prime}). Section~\ref{number} is devoted to
finding a bound for the number of minimal components of an HN
ideal. Concretely, we prove that the number of minimal components of
an HN ideal $I$ is bounded above in terms of the greatest common
divisor of any pair $(m_{i},m_{j})$, $i\neq j$, where
$(m_{1},m_{2},m_{3})=m(I)$ is the integer vector associated to $I$
(Theorem~\ref{bound}). Finally, in Section~\ref{radical}, and provided
that $k$ has characteristic zero or is big enough, we prove that an HN
ideal is always radical (Theorem~\ref{theorem-radical}). We finish by
giving some illustrative examples.

\section{Preliminary results}\label{basic}

In this section, $A$ will be a commutative Noetherian ring and
$x=x_{1},x_{2},x_{3}$ a sequence of elements of $A$ generating a
proper ideal of height 3. We keep the notations of
Section~\ref{introduction}, i.e.,
$v_{1}=x_{1}^{c_{1}}-x_{2}^{b_{2}}x_{3}^{a_{3}}$,
$v_{2}=x_{2}^{c_{2}}-x_{1}^{a_{1}}x_{3}^{b_{3}}$ and
$D=x_{3}^{c_{3}}-x_{1}^{b_{1}}x_{2}^{a_{2}}$, and
$u_{1}=x_{1}^{a_{1}}$ and $u_{2}=-x_{2}^{b_{2}}$, where
$a,b\in\bn^{3}$. In particular, set $I=I_{2}(\cm)=(v_{1},v_{2},D)$.

\begin{Remark}\label{radv}{\rm
Let $I$ be an HN ideal. Then ${\rm rad}(v)={\rm rad}(I)\cap {\rm
  rad}(u)$.  
}\end{Remark}

\begin{proof}
Let $\fp$ a prime ideal of $A$. Suppose that $\fp\supseteq (v)$. Then
$x_{2}^{b_{2}}D=-x_{3}^{b_{3}}v_{1}-x_{1}^{b_{1}}v_{2}\in\fp$. If
$x_{2}\not\in\fp$, then $D\in\fp$ and $I\subseteq \fp$. On the other
hand, if $x_{2}\in\fp$, then $x_{1}\in\fp$ since $v_{1}\in\fp$. It
follows that $\fp\supseteq I$ or $\fp\supseteq (u)$, or equivalently
$\fp\supseteq I\cap (u)$. Conversely, it is immediate that if
$\fp\supseteq I\cap (u)$, then $\fp\supseteq I\supseteq (v)$ or
$\fp\supseteq (u)\supseteq (v)$.
\end{proof}

For stronger results we need more restrictive conditions; see
Proposition~\ref{ifgradev2} $(a)$ and Corollary~\ref{Icapu=v} below.
The following proposition is a direct consequence of the work of
Northcott in \cite{northcott}.

\begin{Proposition}\label{ifgradev2}
Let $I$ be an HN ideal. Suppose that ${\rm grade}(v)=2$.
\begin{itemize}
\item[$(a)$] Then $(v):D=(u)$ and $(v):I=(u)$.
\item[$(b)$] The ideals $(v)$, $I$ and $(u)$ are grade-unmixed of
  grade $2$.
\item[$(c)$] Moreover, if $x=x_{1},x_{2},x_{3}$ is a regular sequence,
  then each of $x_{1},x_{2},x_{3}$ is regular modulo $I$.
\end{itemize}
\end{Proposition}

\begin{proof}
The first part of $(a)$ follows from \cite[Proposition~1]{northcott}
(which does not require the ring to be local), and that $(v):I=(u)$
follows from the equality $(v):D=(u)$. By
\cite[Theorem~130]{kaplansky}, $(v)$ is grade-unmixed and by
\cite[Theorem~2]{northcott}, $I$ also has grade 2 and is
grade-unmixed.  From a minimal primary decomposition of $(v)$ and the
equality $(v):D=(u)$, one can extract another minimal primary
decomposition of $(u)$, and hence $(u)$ is grade-unmixed and clearly
has grade 2. Moreover, if $x=x_{1},x_{2},x_{3}$ is a regular sequence
and if some $x_{i}$ were in an associated prime $\fp$ of $I$, then
$\fp$ would contain $x_{1},x_{2},x_{3}$, a contradiction, since $\fp$
has grade 2 by part $(b)$.
\end{proof}

\begin{Corollary}\label{Icapu=v}
Let $I$ be an HN ideal. Suppose that $x=x_{1},x_{2},x_{3}$ is a
regular sequence and that ${\rm grade}(v)=2$. Then $I\cap (u)=(v)$.
\end{Corollary}

\begin{proof}
Clearly $(v)\subseteq I\cap (u)$ and, by Cramer's rule, $D\cdot
(u)\subset (v)$. To see that $I\cap (u)\subseteq (v)$, it suffices to
show $(u):D=(u)$, for if $z\in I\cap (u)$, then $z=w+qD\in (u)$, with
$w\in (v)\subseteq (u)$. Thus $qD\in (u)$ and $q\in (u):D=(u)$ would
follow. Therefore one would deduce that $qD\in (v)$ and $z\in
(v)$. But $D$ is indeed regular modulo $(u)$. Were $D$ in an
associated prime $\fp$ of $(u)$, then $\fp$ would contain the regular
sequence $x_{1},x_{2},x_{3}$, a contradiction, since $\fp$ has grade 2
by Proposition~\ref{ifgradev2} $(b)$.
\end{proof}

\section{HN ideals are set theoretically complete 
intersections}\label{stci}

In this section, $A$ will be a commutative Noetherian ring and
$x=x_{1},x_{2},x_{3}$ a sequence of elements of $A$ generating a
proper ideal of height 3. We keep the notations of
Section~\ref{introduction}, i.e.,
$v_{1}=x_{1}^{c_{1}}-x_{2}^{b_{2}}x_{3}^{a_{3}}$,
$v_{2}=x_{2}^{c_{2}}-x_{1}^{a_{1}}x_{3}^{b_{3}}$ and
$D=x_{3}^{c_{3}}-x_{1}^{b_{1}}x_{2}^{a_{2}}$, and
$u_{1}=x_{1}^{a_{1}}$ and $u_{2}=-x_{2}^{b_{2}}$, with
$a,b\in\bn^{3}$.

In the next result we produce the desired element $g$ using an
algorithm employed by Bresinsky in \cite{bresinsky} (there in the case
where $A=k[x_{1},x_{2},x_{3}]$ is a polynomial ring over the field
$k$).  For this reason we will refer to this choice of $g$ as $g_{B}$.
This candidate for $g$ will subsequently be contrasted in
Example~\ref{gbandgv} and Remark~\ref{gbvsgv} below with the candidate
for $g$, denoted $g_{V}$, advanced by Valla in \cite{valla} (in our
general setting). To examine this contrast in detail, we pay close
attention to the form of $g$.

\begin{Theorem}\label{gb} (cf. \cite{bresinsky}, \cite{valla}.)
Let $I$ be an HN ideal. Then there exists an algorithmically specified
element $g$ in the radical of $I$ such that ${\rm rad}(I)={\rm
  rad}(v_{1},g)$. Moreover, if $x=x_{1},x_{2},x_{3}$ is a regular
sequence and ${\rm grade}(v)=2$, then this element $g$ lies in $I$.
\end{Theorem}

\begin{proof}
We prove that there exists an element $g$ in $A$ of the form
$g=(-1)^{c_{1}}x_{3}^{r}+h$, with $r\geq 1$ and $h\in (x_{1},x_{2})$,
and such that $v_{2}^{c_{1}}-pv_{1}=x_{2}^{a_{1}b_{2}}g$, for some
$p\in A$.

From this last equation, it follows that $v_{2}\in {\rm rad}(v_{1},g)$
and $g\in (v):u_{2}^{a_{1}}\subseteq I:u_{2}^{a_{1}}$. Since $I$ is
the ideal generated by the $2\times 2$ minors of a $2\times 3$ matrix,
by \cite[Theorem~3]{en} any minimal prime of $I$ is of height at most
2.  In particular, $x_{2}$ is not in any minimal prime of $I$. Thus
$u_{2}^{a_{1}}$ is regular modulo ${\rm rad}(I)$ and $g\in {\rm
  rad}(I)$. Moreover, if $x=x_{1},x_{2},x_{3}$ is a regular sequence
and ${\rm grade}(v)=2$, by Proposition~\ref{ifgradev2}, $x_{2}$ and
hence $u_{2}^{a_{1}}$ is regular modulo $I$ and $g\in I$ in this
case. Then $D^{r}-(-1)^{c_{1}c_{3}}g^{c_{3}}\in {\rm rad}(I)\cap
(x_{1},x_{2})\subseteq {\rm rad}(I\cap (u))$ which, by
Remark~\ref{radv}, is equal to ${\rm rad}(v)$ and so included in ${\rm
  rad}(v_{1},g)$. Thus $D\in {\rm rad}(v_{1},g)$ and ${\rm
  rad}(I)={\rm rad}(v_{1},g)$, and the result is then proved.

Now construct the desired $g$ (call it $g_{B}$) by following
Bresinsky's argument in \cite{bresinsky}. His idea is to take the
binomial expansion of
$v_{2}^{c_{1}}=(x_{2}^{c_{2}}-x_{1}^{a_{1}}x_{3}^{b_{3}})^{c_{1}}$
and, by subtracting a multiple of
$v_{1}=x_{1}^{c_{1}}-x_{2}^{b_{2}}x_{3}^{a_{3}}$, eliminate the higher
order terms in $x_{1}^{a_{1}}x_{3}^{b_{3}}$. To assist the reader we
include details that were omitted in Bresinsky's paper.

For $i=0$, write $B_{0}=v_{2}^{c_{1}}$, $p_{0}=0$ and $q_{0}=0$. For
$i=1,\ldots ,a_{1}$, write

\begin{eqnarray*}
&&B_{i}=\sum_{j=i}^{c_{1}}(-1)^{c_{1}-j}\binom{c_{1}}{j}
  x_{2}^{jc_{2}}x_{1}^{a_{1}(c_{1}-j)}x_{3}^{b_{3}(c_{1}-j)},
  \\&&p_{i}=\sum_{j=0}^{i-1}(-1)^{c_{1}-j}\binom{c_{1}}{j}
  x_{2}^{ja_{2}+(i-1)b_{2}}x_{1}^{(a_{1}-i+j)c_{1}-ja_{1}}
  x_{3}^{(i-1-j)a_{3}+b_{3}(c_{1}-j)},
\end{eqnarray*}
\begin{eqnarray*}
&&q_{i}=p_{i}x_{2}^{b_{2}}x_{3}^{a_{3}}=
  \sum_{j=0}^{i-1}(-1)^{c_{1}-j}\binom{c_{1}}{j}
  x_{2}^{ja_{2}+ib_{2}}x_{1}^{(a_{1}-i+j)c_{1}-ja_{1}}
  x_{3}^{(i-j)a_{3}+b_{3}(c_{1}-j)}.
\end{eqnarray*}

Note that $B_{i}$ is just a part of the binomial expansion of
$v_{2}^{c_{1}}=(x_{2}^{c_{2}}-x_{1}^{a_{1}}x_{3}^{b_{3}})^{c_{1}}$. Moreover,

\begin{eqnarray*}
&&B_{i}-B_{i+1}+q_{i}-p_{i+1}x_{1}^{c_{1}}=
  B_{i}-B_{i+1}+p_{i}x_{2}^{b_{2}}x_{3}^{a_{3}}-p_{i+1}x_{1}^{c_{1}}
  \\&&=(-1)^{c_{1}-i}\binom{c_{1}}{i}
  x_{2}^{ic_{2}}x_{1}^{a_{1}(c_{1}-i)}x_{3}^{b_{3}(c_{1}-i)}
  \\&&+\left[ \sum_{j=0}^{i-1}(-1)^{c_{1}-j}\binom{c_{1}}{j}
    x_{2}^{ja_{2}+(i-1)b_{2}}x_{1}^{(a_{1}-i+j)c_{1}-ja_{1}}
    x_{3}^{(i-1-j)a_{3}+b_{3}(c_{1}-j)}\right]x_{2}^{b_{2}}x_{3}^{a_{3}}
  \\&&-\left[\sum_{j=0}^{i}(-1)^{c_{1}-j}\binom{c_{1}}{j}
    x_{2}^{ja_{2}+ib_{2}}x_{1}^{(a_{1}-i-1+j)c_{1}-ja_{1}}
    x_{3}^{(i-j)a_{3}+b_{3}(c_{1}-j)}\right]x_{1}^{c_{1}}=0.
\end{eqnarray*}

Now, for $i=0,\ldots ,a_{1}$, let us see that
$v_{2}^{c_{1}}-\sum_{j=0}^{i}p_{j}v_{1}=B_{i}+q_{i}$. Indeed, for
$i=0$, since $B_{0}=v_{2}^{c_{1}}$ and $p_{0}=0$ and $q_{0}=0$,
$v_{2}^{c_{1}}-p_{0}v_{1}=B_{0}+q_{0}$. Suppose the equality holds for
$i$, $0<i<a_{1}$. Then
\begin{eqnarray*}
&&v_{2}^{c_{1}}-\sum_{j=0}^{i+1}p_{j}v_{1}=B_{i}+q_{i}-p_{i+1}v_{1}=
  \\&&B_{i+1}+(B_{i}-B_{i+1}+q_{i}-p_{i+1}x_{1}^{c_{1}})+
  p_{i+1}x_{2}^{b_{2}}x_{3}^{a_{3}}=B_{i+1}+q_{i+1}.
\end{eqnarray*}
In particular, for $i=a_{1}$, 
\begin{eqnarray*}
&&v_{2}^{c_{1}}-\sum_{j=0}^{a_{1}}p_{j}v_{1}=B_{a_{1}}+q_{a_{1}}=
  \\&&\sum_{j=a_{1}}^{c_{1}}(-1)^{c_{1}-j}\binom{c_{1}}{j}
  x_{2}^{jc_{2}}x_{1}^{a_{1}(c_{1}-j)}x_{3}^{b_{3}(c_{1}-j)}+
  \\&&\sum_{j=0}^{a_{1}-1}(-1)^{c_{1}-j}\binom{c_{1}}{j}
  x_{2}^{ja_{2}+a_{1}b_{2}}x_{1}^{j(c_{1}-a_{1})}
  x_{3}^{(a_{1}-j)a_{3}+b_{3}(c_{1}-j)}= \\
&&x_{2}^{a_{1}b_{2}}\left[
    \sum_{j=a_{1}}^{c_{1}}(-1)^{c_{1}-j}\binom{c_{1}}{j}
    x_{2}^{ja_{2}+(j-a_{1})b_{2}}x_{1}^{a_{1}(c_{1}-j)}x_{3}^{b_{3}(c_{1}-j)}\right]
  + \\&&x_{2}^{a_{1}b_{2}}\left[
    \sum_{j=0}^{a_{1}-1}(-1)^{c_{1}-j}\binom{c_{1}}{j}
    x_{2}^{ja_{2}}x_{1}^{j(c_{1}-a_{1})}
    x_{3}^{(a_{1}-j)a_{3}+b_{3}(c_{1}-j)}\right].
\end{eqnarray*}

Thus, taking $p=\sum_{j=0}^{a_{1}}p_{j}v_{1}$ and $g_{B}$ the element
given as follows:

\begin{eqnarray*}
g_{B}=&&\sum_{j=a_{1}}^{c_{1}}(-1)^{c_{1}-j}\binom{c_{1}}{j}
x_{2}^{ja_{2}+(j-a_{1})b_{2}}x_{1}^{a_{1}(c_{1}-j)}x_{3}^{b_{3}(c_{1}-j)}
\\ +&&\sum_{j=0}^{a_{1}-1}(-1)^{c_{1}-j}\binom{c_{1}}{j}
x_{2}^{ja_{2}}x_{1}^{j(c_{1}-a_{1})}x_{3}^{(a_{1}-j)a_{3}+b_{3}(c_{1}-j)},
\end{eqnarray*}
one has $v_{2}^{c_{1}}-pv_{1}=x_{2}^{a_{1}b_{2}}g_{B}$. Moreover, note
that the term $j=0$ of $g_{B}$ gives $(-1)^{c_{1}}x_{3}^{r}$, with
$r=a_{1}a_{3}+b_{3}c_{1}$, and the rest of the terms of $g_{B}$ are in
$(x_{1},x_{2})$.
\end{proof}

\begin{Example}\label{gbandgv}{\rm 
In \cite[Theorem~3]{valla}, Valla constructed an element $g_{V}$ in
the radical of $I$ such that ${\rm rad}(I)={\rm rad}(v_{1},g_{V})$. In
concrete terms, changing Valla's notation to ours,
\begin{eqnarray*}
g_{V}=\sum_{j=0}^{c_{1}}(-1)^{c_{1}-j}\binom{c_{1}}{j}x_{1}^{s}
x_{2}^{(c_{1}-j)c_{2}+tb_{2}-a_{1}b_{2}}x_{3}^{jb_{3}+ta_{3}},
\end{eqnarray*}
where for $0\leq j\leq c_{1}$, $ja_{1}=tc_{1}+s$ with $0\leq s\leq
c_{1}-1$.

For instance, if we take $\cm$ to be as follows:
\begin{eqnarray*}
\cm=\left(\begin{array}{ccc}x_{1}^{2}&x_{2}^{}&x_{3}^{}
  \\ x_{2}^{2}&x_{3}^{}&x_{1}^{2}\end{array}\right),
\end{eqnarray*}
which is the example considered by Bresinsky in \cite{bresinsky}, then
the element $g_{B}$ considered in the proof of Theorem~\ref{gb} is as
follows:
\begin{eqnarray*}
g_{B}=\sum_{j=2}^{4}(-1)^{4-j}\binom{4}{j}
x_{2}^{j+(j-2)2}x_{1}^{2(4-j)}x_{3}^{4-j}
+\sum_{j=0}^{1}(-1)^{4-j}\binom{4}{j}
x_{2}^{j}x_{1}^{j2}x_{3}^{(2-j)+(4-j)},
\end{eqnarray*}
so that,
\begin{eqnarray*}
g_{B}=x_{2}^{8}-4x_{2}^{5}x_{1}^{2}x_{3}+6x_{2}^{2}x_{1}^{4}x_{3}^{2}-
4x_{2}x_{1}^{2}x_{3}^{4}+x_{3}^{6}.
\end{eqnarray*}
On the other hand,
\begin{eqnarray*}
g_{V}=x_{2}^{8}-4x_{2}^{5}x_{1}^{2}x_{3}+6x_{2}^{4}x_{3}^{3}-
4x_{2}x_{1}^{2}x_{3}^{4}+x_{3}^{6}.
\end{eqnarray*}
Remark that $g_{B}-g_{V}=6x_{2}^{2}x_{3}^{2}v_{1}$. This pattern is
completely general, as we now show.  }\end{Example}

\begin{Remark}\label{gbvsgv}{\rm 
If $v_{1},x_{2}$ is a regular sequence, then $g_{B}-(-1)^{c_{1}}g_{V}$
is in $(v_{1})$. In particular, if moreover $x=x_{1},x_{2},x_{3}$ is a
regular sequence and ${\rm grade}(v)=2$, then $g_{B}$ and $g_{V}$ are
in $I$.  }\end{Remark}

\begin{proof}
By the beginning of the proof of Valla \cite[Theorem~3]{valla} -
changing his notation to ours - we have:
\begin{eqnarray*}
(-1)^{c_{1}}v_{2}^{c_{1}}=x_{2}^{a_{1}b_{2}}g_{V}\mbox{ mod }(v_{1}).
\end{eqnarray*}
On the other hand, from Theorem~\ref{gb}, we have
$v_{2}^{c_{1}}=x_{2}^{a_{1}b_{2}}g_{B}+pv_{1}$. Hence
\begin{eqnarray}\label{gb-gv}
x_{2}^{a_{1}b_{2}}[g_{B}-(-1)^{c_{1}}g_{V}]=0\mbox{ mod }(v_{1}).
\end{eqnarray}
If $x_{2}$ is regular modulo $(v_{1})$, then
$g_{B}-(-1)^{c_{1}}g_{V}\in (v_{1})$. If moreover
$x=x_{1},x_{2},x_{3}$ is a regular sequence and ${\rm grade}(v)=2$, by
Theorem~\ref{gb}, $g_{B}$ can be taken in $I$, and so can $g_{V}$.
\end{proof}

\begin{Remark}{\rm Note that it is always the case that $x_{2}$ is 
regular modulo ${\rm rad}(I)$ for a general HN ideal $I$ (see the
proof of Theorem~\ref{gb}). Hence it follows from Theorem~\ref{gb} and
equation~(\ref{gb-gv}) above that ${\rm rad}(I)={\rm
  rad}(v_{1},g_{V})$ for a general HN ideal $I$. }\end{Remark}

\section{On the condition ${\rm grade}(v)=2$}\label{secgradev=2}

In this section, $A$ will be a commutative Noetherian ring and
$x=x_{1},x_{2},x_{3}$ a sequence of elements of $A$ (not necessarily
generating a proper ideal of height 3). We keep the notations of
Section~\ref{introduction}, i.e.,
$v_{1}=x_{1}^{c_{1}}-x_{2}^{b_{2}}x_{3}^{a_{3}}$,
$v_{2}=x_{2}^{c_{2}}-x_{1}^{a_{1}}x_{3}^{b_{3}}$ and
$D=x_{3}^{c_{3}}-x_{1}^{b_{1}}x_{2}^{a_{2}}$. The purpose of this
section is to study the condition ${\rm grade}(v) = 2$ versus the
condition that $x_{1},x_{2}$ or $x_{1},x_{2},x_{3}$ forms a regular
sequence. These results are of interest in view of Remark~\ref{gbvsgv}
and of results in subsequent sections.

In the latter half of the paper, we shall be particularly interested
in the case of a polynomial ring. As will be seen in
Example~\ref{as*}, this case can be placed in a graded context. For
this reason and for ease of reference, we introduce the following
notation. With the assumptions of this section, we say that $(A,x)$
satisfies the {\em homogeneous condition} $(*)$ if $A$ can be graded
by $\bn_{0}$, with $x_{1},x_{2},x_{3}$ and $v_{1},v_{2},D$ homogeneous
elements of positive degree.

The following result is folklore; see \cite[Corollary~1.6.19]{bh} for
the unexplained notations and a proof in the local case. For the
graded case use \cite[\S~9.7, Corollaire~2]{bourbaki} and
\cite[Theorem~1.6.17 $(b)$]{bh}.

\begin{Theorem}\label{folklore}
Let $R$ be a commutative Noetherian ring. Let $y=y_{1},\ldots ,y_{n}$
be a sequence of elements of $R$. Let $M$ be a finitely generated
$R$-module. Suppose either that $y_{1},\ldots ,y_{n}$ are in the
Jacobson radical of $R$, or that $R=\oplus_{n\geq 0}R_{n}$ is
$\bn_{0}$-graded, $y_{1},\ldots ,y_{n}$ are homogeneous elements of
positive degree and $M=\oplus_{n\geq 0}M_{n}$ is graded over $R$.
Then the following conditions are equivalent:
\begin{itemize}
\item[$(1)$] ${\rm grade}((y_{1},\ldots ,y_{n});M)=n$;
\item[$(2)$] $H_{i}(y;M)=0$ for all $i\geq 1$;
\item[$(3)$] $H_{1}(y;M)=0$;
\item[$(4)$] $y=y_{1},\ldots ,y_{n}$ is a regular sequence (in any
  order).
\end{itemize}
\end{Theorem}

We state now the desired result.

\begin{Proposition}\label{grade2}
Let $x=x_{1},x_{2},x_{3}$ be a sequence of elements of $A$. Suppose
either that $x_{1},x_{2},x_{3}$ are in the Jacobson radical of $A$, or
that $(A,x)$ satisfies the homogeneous condition $(*)$. Then
\begin{itemize}
\item[$(a)$] $x_{1},x_{2}$ is a regular sequence if and only if
  $v_{1},x_{2}$ is a regular sequence;
\item[$(b)$] $x_{1},x_{2},x_{3}$ is a regular sequence if and only if
  $v_{1},v_{2},x_{3}$ is a regular sequence (so ${\rm grade}(v)=2$ in
  either case);
\item[$(c)$] If moreover $x_{1},x_{2},x_{3}$ generate a proper ideal
  of height 3, ${\rm grade}(v)=2$ and $A$ satisfies the Serre
  condition $(S_{3})$, then $x_{1},x_{2},x_{3}$ is a regular sequence.
\end{itemize}
\end{Proposition}

\begin{proof}
We have $(v_{1},x_{2})=(x_{1}^{c_{1}},x_{2})$ and
$(v_{1},v_{2},x_{3})=(x_{1}^{c_{1}},x_{2}^{c_{2}},x_{3})$. Thus ${\rm
  grade}(v_{1},x_{2})={\rm grade}(x_{1}^{c_{1}},x_{2})$ and ${\rm
  grade}(v_{1},v_{2},x_{3})={\rm
  grade}(x_{1}^{c_{1}},x_{2}^{c_{2}},x_{3})$. By
Theorem~\ref{folklore}, $v_{1},x_{2}$ is a regular sequence if and
only if ${\rm grade}(v_{1},x_{2})=2$, and so, by
Theorem~\ref{folklore} again, if and only if $x_{1}^{c_{1}},x_{2}$ is
a regular sequence. Analogously, $v_{1},v_{2},x_{3}$ is a regular
sequence if and only if $x_{1}^{c_{1}},x_{2}^{c_{2}},x_{3}$ is a
regular sequence. Using \cite[Exercise~3.1.12 $(c)$]{kaplansky}, one
deduces $(a)$ and $(b)$.

Finally, if ${\rm grade}(v)=2$, since $(v)\subseteq (x_{1},x_{2})$,
${\rm grade}(x_{1},x_{2})=2$, by \cite[Theorem~125]{kaplansky}. By
Theorem~\ref{folklore}, $x_{1},x_{2}$ is a regular sequence, and by
\cite[Theorem~130]{kaplansky}, the ideal $(x_{1},x_{2})$ is
grade-unmixed. Thus, for any associated prime $\fp$ of
$(x_{1},x_{2})$, ${\rm depth}(A_{\mathfrak{p}})={\rm grade}(\fp)=2$
and ${\rm height}(\fp)\geq 2$. Since $A$ satisfies the Serre condition
$(S_{3})$, $2=\mbox{depth}(A_{\mathfrak{p}})\geq \inf (3,
\mbox{height}(\fp))$. Hence $\fp$ has height $2$, so $x_{3}\notin \fp$
since $(x_{1},x_{2},x_{3})$ has height $3$, by assumption. Thus
$x_{1},x_{2},x_{3}$ is a regular sequence.
\end{proof}

\begin{Example}\label{as*} {\rm Let $A=k[x_{1},x_{2},x_{3}]$
be the polynomial ring with $x_{1},x_{2},x_{3}$ three variables over a
field $k$. Set $m_{1}=c_{2}c_{3}-a_{2}b_{3}$,
$m_{2}=c_{1}c_{3}-a_{3}b_{1}$ and $m_{3}=c_{1}c_{2}-a_{1}b_{2}$. Endow
$A$ with the natural grading induced by giving $x_{i}$ weight
$m_{i}$. Then $A$ is graded by $\bn_{0}$ and $x_{1},x_{2},x_{3}$ and
$v_{1},v_{2},D$ are homogeneous elements of positive degree, i.e.,
$(A,x)$ satisfies the homogeneous condition $(*)$. In particular,
$v_{1},v_{2}$ is a regular sequence in either order.  }\end{Example}

\begin{proof}
Clearly $m_{i}>0$, so $x_{i}$ is a homogeneous element of positive
degree. On the other hand, $v_{1},v_{2},D$ are homogeneous provided
that $(m_{1},m_{2},m_{3})$ satisfies the following system of
equations:
\begin{eqnarray}\label{m-system1}
\left\{ \begin{array}{ccc}c_{1}m_{1}&=&b_{2}m_{2}+a_{3}m_{3},\\
c_{2}m_{2}&=&a_{1}m_{1}+b_{3}m_{3},\\
c_{3}m_{3}&=&b_{1}m_{1}+a_{2}m_{2}.\end{array} \right.
\end{eqnarray}
It is easily checked that this is indeed the case.  Hence
$v_{1},v_{2},D$ are homogeneous elements of positive degree and so
$(A,x)$ satisfies the homogeneous condition $(*)$. Applying
Proposition~\ref{grade2}, we deduce that $v_{1},v_{2}$ is a regular
sequence, and in either order because of Theorem~\ref{folklore}.
\end{proof}

\begin{Remark}\label{unique-sol}{\rm 
In fact, note that in the system of equations~(\ref{m-system1}), each
one of the three equations can be obtained from the other two via
addition. Moreover, since $c_{1}>a_{1}$ and $c_{2}>b_{2}$, the system
can be reduced to the $\bq$-linear system of rank 2 formed by the two
first equations:
\begin{eqnarray*}
\left( \begin{array}{rrr}c_{1}&-b_{2}&-a_{3}\\-a_{1}&c_{2}&-b_{3}\end{array}
\right)\left( \begin{array}{c}m_{1}\\m_{2}\\m_{3}\end{array}\right)=
\left( \begin{array}{c}0\\0\end{array}\right).
\end{eqnarray*}
By Cramer's rule, the $\bq$-linear subspace of solutions is generated
by the non-zero vector
\begin{eqnarray*}
&&\left(
\left| \begin{array}{rr}a_{3}&-b_{2}\\b_{3}&c_{2}
\end{array}\right|,
\left| \begin{array}{rr}c_{1}&a_{3}\\-a_{1}&b_{3}\end{array}\right|,
\left| \begin{array}{rr}c_{1}&-b_{2}\\-a_{1}&c_{2}
\end{array}\right|\right)=\\
&&(a_{2}a_{3}+a_{3}b_{2}+b_{2}b_{3},a_{1}a_{3}+a_{1}b_{3}+b_{1}b_{3},
a_{1}a_{2}+a_{2}b_{1}+b_{1}b_{2})=
\\ &&(c_{2}c_{3}-a_{2}b_{3},c_{1}c_{3}-a_{3}b_{1},c_{1}c_{2}-a_{1}b_{2})
=(m_{1},m_{2},m_{3}).
\end{eqnarray*}
Thus $(m_{1},m_{2},m_{3})\in\bn^{3}$ is a solution of
system~(\ref{m-system1}) which is unique up to a non-zero rational
multiple.  
}\end{Remark}

One can have a Cohen-Macaulay domain $A$ and a sequence of elements
$x=x_{1},x_{2},x_{3}$ of $A$ generating a proper ideal of height, and
hence grade, 3, though with ${\rm grade}(v_{1},v_{2})$ equal to 1.
Necessarily, by Proposition~\ref{grade2}, $x=x_{1},x_{2},x_{3}$ are
not in the Jacobson radical and $(A,x)$ does not satisfy the
homogeneous condition $(*)$.

\begin{Example}{\rm (see \cite[Exercise~7, p.~102]{kaplansky})
Let $A=k[y_{1},y_{2},y_{3}]$ be the polynomial ring with
$y_{1},y_{2},y_{3}$ three variables over a field $k$. Then the
elements $x_{3}=y_{1}$, $x_{1}=y_{2}(1-y_{1})$ and
$x_{2}=y_{3}(1-y_{1})$ form a regular sequence in this order, but in
the order $x_{1},x_{2},x_{3}$ they do not; moreover ${\rm
  grade}(x_{1},x_{2})={\rm grade}((x_{1},x_{2});A/(x_{1}))+1=1$. Thus
$x_{1},x_{2},x_{3}$ generate a proper ideal of height (and grade) $3$,
but ${\rm grade}(v_{1},v_{2})=1$.  }\end{Example}

\section{HN ideals are geometrically linked to complete 
intersections}\label{licci}

In this section, $A$ will be a commutative Noetherian ring and
$x=x_{1},x_{2},x_{3}$ a regular sequence.  Moreover, we will suppose
that the ideal $(v)$ has grade 2 (see
e.g. Proposition~\ref{grade2}). In spite of $(v)$ having grade 2, it
may happen that $v_{1},v_{2}$ is not a regular sequence. However one
can ensure that there does exist an element $w\in A$ such that
$v_{1}+wv_{2},v_{2}$ is a regular sequence (see
e.g. \cite[Theorem~125]{kaplansky} and its proof). We keep the rest of
the notations as in Section~\ref{introduction}.

\begin{Lemma}\label{ivu1}
Let $I$ be an HN ideal. Suppose that $x=x_{1},x_{2},x_{3}$ is a
regular sequence and that ${\rm grade}(v)=2$. Then
$I=(v):u_{1}=(v):u_{2}=(v):(u)$.
\end{Lemma}

\begin{proof} 
By Corollary~\ref{Icapu=v}, $I\cap (u)=(v)$ and, by
Proposition~\ref{ifgradev2}, $x_{1}$ is regular modulo $I$. Then
\begin{eqnarray*}
(v):u_{1}=[I\cap (u)]:u_{1}=(I:u_{1})\cap [(u):u_{1}]=I:u_{1}=I.
\end{eqnarray*}
Analogously $I=(v):u_{2}$. That $I=(v):(u)$ follows immediately from
this, since $(v):(u)=((v):u_{1})\cap ((v):u_{2})$.
\end{proof}

In particular, $(v)$ is a radical ideal if and only if $I$ and $(u)$
are radical ideals. Indeed, by Remark~\ref{radv}, ${\rm rad}(v)={\rm
  rad}(I)\cap {\rm rad}(u)$ and so $(v)$ is radical if $I$ and $(u)$
are radical. Conversely, if $(v)$ is radical, then ${\rm rad}(I)={\rm
  rad}((v):u_{1})\subseteq {\rm rad}(v):u_{1}=(v):u_{1}=I$ and ${\rm
  rad}(u)={\rm rad}((v):D)\subseteq {\rm rad}(v):D=(v):D=(u)$, by
Proposition~\ref{ifgradev2}. In particular, if $a_{1}>1$ or $b_{2}>1$,
then $(u)$, and hence $(v)$, is not radical.

\begin{Proposition}
Let $I$ be an HN ideal. Suppose that $x=x_{1},x_{2},x_{3}$ is a
regular sequence and that ${\rm grade}(v)=2$. Then $I$ is
geometrically linked to $(u)$, i.e., $(u)=(v):I$, $I=(v):(u)$ and
$I\cap (u)=(v)$.
\end{Proposition}

\begin{proof}
By \cite[Theorem~125]{kaplansky}, $(v)$ can be generated by a regular
sequence clearly contained in $I\cap (u)$. Moreover, that $(u)=(v):I$,
$I=(v):(u)$ and $I\cap (u)=(v)$ follow from
Proposition~\ref{ifgradev2}, Lemma~\ref{ivu1} and Corollary~\ref{Icapu=v},
respectively.
\end{proof}

In \cite[page 326 and ff.]{vasconcelos}, Vasconcelos gives a proof
that $I=(v):(u)$, but it would seem that there is a hidden Gorenstein
hypothesis in his Corollary~4.1.1 (see the appeal to Corollary~A.9.1
in its proof; see also \cite[Proposition~2.4]{ulrich}, where the local
Gorenstein hypothesis is used again).

In particular, one has ${\rm Ass}(A/I)\cup {\rm Ass}(A/(u))={\rm
  Ass}(A/(v))$ (see e.g. \cite[Remark~2.2]{ulrich}).

\section{HN ideals are almost complete intersections}\label{aci}

In this section, $A$ will again be a commutative Noetherian ring and
$x=x_{1},x_{2},x_{3}$ a regular sequence. Moreover, as before, we will
suppose that the ideal $(v)$ has grade 2 (see
Proposition~\ref{grade2}).

\begin{Lemma}\label{mui=3}
Let $I$ be an HN ideal. Suppose that $x=x_{1},x_{2},x_{3}$ is a
regular sequence and that ${\rm grade}(v)=2$. Then $I$ is minimally
generated by three elements.
\end{Lemma}

\begin{proof}
By Proposition~\ref{ifgradev2}, none of $x_{1}$, $x_{2}$ or $x_{3}$ is
in any minimal prime of $I$. Localise at a minimal prime of
$(x_{1},x_{2},x_{3})$ without changing notations. So we suppose that
$A$ is local. If $I$ has a minimal generating set of less than three
elements, then at least one element of the generating set
$v_{1},v_{2},D$ is redundant, $D$ say. In this case,
$I=(v_{1},v_{2})\subseteq (x_{1},x_{2})$. By the Generalized Principal
Ideal Theorem (\cite[Theorem~152]{kaplansky}), there exists a minimal
prime $\fq$ of $(x_{1},x_{2})$ of height 2. But then $\fq$ would be a
minimal prime over $I$ containing $(x_{1},x_{2},x_{3})$, a
contradiction (and similarly for the possible variations on this
argument).

Alternatively, localise at a minimal prime containing
$(x_{1},x_{2},x_{3})$ without changing notations. The resolving
complex constructed in \cite[Section~2 and Theorem~2]{northcott} in
the case of our specific $\Phi$, $u$ and $v$ yields the following free
resolution of $I$:
\begin{eqnarray*}
0\rightarrow A^{2}\buildrel [\varphi]\over\rightarrow A^{3}\buildrel
[\psi]\over\rightarrow A\rightarrow A/(v,D)\rightarrow 0,
\end{eqnarray*}
where
\begin{eqnarray*}
[\varphi]=\left( \begin{array}{rr}
-x_{1}^{a_{1}}&x_{2}^{b_{2}}\\
x_{2}^{a_{2}}&-x_{3}^{b_{3}}\\
x_{3}^{a_{3}}&-x_{1}^{b_{1}}
\end{array}\right)\mbox{ and }
[\psi]=\left( \begin{array}{ccc}-D&-v_{1}&v_{2}\end{array}\right).
\end{eqnarray*}
Since all the entries in the matrix maps are in the maximal ideal,
this Hilbert-Burch presentation of $I$ is minimal. 
\end{proof}

\begin{Remark}\label{aciingeneral}
{\rm Actually, by Valla's argument at the top of page~10 in
  \cite{valla}, Lemma~\ref{mui=3} holds for an arbitrary HN
  ideal. Note also that the second proof presented above requires only
  that ${\rm grade}(v)=2$. Remark too that from this resolution,
  $[\psi]\cdot [\varphi]=0$ and so
  $Dx_{1}^{a_{1}}-v_{1}x_{2}^{a_{2}}+v_{2}x_{3}^{a_{3}}=0$. Therefore
  $Dx_{1}^{a_{1}}\in (v)$ and $IA_{x_{1}}=(v)A_{x_{1}}$. (This can
  also be deduced from the equality $I=(v):(u)$.)  }\end{Remark}

\begin{Proposition}\label{proposition-aci}
Let $I$ be an HN ideal. Suppose that $x=x_{1},x_{2},x_{3}$ is a
regular sequence and that ${\rm grade}(v)=2$. Then $I$ is an almost
complete intersection (in the sense of \cite{hmv}).
\end{Proposition}

\begin{proof}
On the one hand, $I$ is minimally generated by 3 elements and has
height 2. To see this, note that $I$ contains $(v)$ so has grade at
least 2.  On the other hand, since ${\rm rad}(I)={\rm rad}(v_{1},g)$,
any minimal prime of $I$ is a minimal prime of $(v_{1},g)$, which by
the Generalized Principal Ideal Theorem will be of height at most
2. Finally, $IA_{\mathfrak{p}}$ is locally a complete intersection at
primes $\fp$ minimal over $I$, because such a prime $\fp$ fails to
contain $x_{1}$, so $IA_{\mathfrak{p}}\cong
(IA_{x_{1}})A_{\mathfrak{p}_{x_{1}}}$. But $IA_{x_{1}}=(v)A_{x_{1}}$
and $IA_{\mathfrak{p}}=(v)A_{\mathfrak{p}}$.
\end{proof}

\begin{Remark}{\rm 
Let $I$ be an HN ideal. Suppose that $x=x_{1},x_{2},x_{3}$ is a
regular sequence and that ${\rm grade}(v)=2$. Then $I$ is generated by
a $d$-sequence, $I$ is of linear type and of strong linear
type. }\end{Remark}

\begin{proof}
Indeed, since $(v):D=(u)$ and $I\cap (u)=(v)$, then $((v):D)\cap
I=(u)\cap I=(v)$. Therefore $(v):D^{2}=(v):D$. In particular, for some
$w\in A$, $v_{1}+wv_{2},v_{2},D$ is a $d$-sequence which generates $I$
(cf. \cite{huneke2} for a general discussion of $d$-sequences, and
[loc. cit., Example~4] in particular for the result at issue here).
Hence $I$ is an ideal of linear type, i.e., the canonical graded
homomorphism $\alpha:{\bf S}(I)\rightarrow {\bf R}(I)$ between the
symmetric algebra of $I$ and the Rees algebra of $I$ is an isomorphism
(see e.g. \cite{hmv}). In fact, one has a little more, namely $I$ is
of strong linear type, i.e., $H_{2}(A,B,{\bf G}(I))=0$, where
$H_{2}(A,B,{\bf G}(I))$ stands for the second Andr\'e-Quillen homology
group of the $A$-algebra $B=A/I$ with coefficients in the $B$-module
${\bf G}(I)$, the associated graded ring of $I$. It is well-known that
${\rm ker}(\alpha_{2})\cong H_{2}(A,B,B)$. Moreover, if $A\supset \bq$
and $H_{2}(A,B,{\bf G}(I))=0$, then $I$ is of linear type. Since $B$
as an $A$-module has projective dimension 2 and $I$ is of linear type,
then the converse also holds and one has $H_{2}(A,B,{\bf G}(I))=0$
(see \cite[Theorem~3.4 and Proposition~3.10]{planas}).
\end{proof}

\section{When are HN ideals prime?}\label{prime}

From now on until the end of the paper, $A=k[x_{1},x_{2},x_{3}]$ is
the polynomial ring with $x=x_{1},x_{2},x_{3}$ three variables over a
field $k$. In particular, $v_{1},v_{2}$ is a regular sequence in any
order, by Example~\ref{as*}. Denote by $\fm$ the maximal ideal of $A$
generated by $x_{1},x_{2},x_{3}$.

We begin with the following definition which will play a key role. The
idea behind it lies in the subsequent remark and in the proofs of
Example~\ref{as*} and Lemma~\ref{unique}.

\begin{Definition}{\rm 
Let $J$ be an ideal of $A$, $J\subset \fm$, such that $x_{i}A+J$ is
$\fm$-primary for all $i=1,2,3$. The {\em integer vector associated
  to} $J$, $m(J)=(m_{1}(J),m_{2}(J),m_{3}(J))\in\bn^{3}$, is defined
as $m_{i}(J)={\rm length}(A/(x_{i}A+J))$, for each $i=1,2,3$.
}\end{Definition}

\begin{Remark}\label{edei}{\rm 
Let $I=(x_{1}^{c_{1}}-x_{2}^{b_{2}}x_{3}^{a_{3}},
x_{2}^{c_{2}}-x_{1}^{a_{1}}x_{3}^{b_{3}},x_{3}^{c_{3}}-x_{1}^{b_{1}}x_{2}^{a_{2}})$
be an HN ideal. Then, for each $i=1,2,3$, $x_{i}A+I$ is $\fm$-primary
and $m_{i}(I)=e(x_{i}\cdot
A_{\mathfrak{m}}/I_{\mathfrak{m}})$. Moreover $m(I)$ can be directly
obtained as
$m(I)=(c_{2}c_{3}-a_{2}b_{3},c_{1}c_{3}-a_{3}b_{1},c_{1}c_{2}-a_{1}b_{2})$.
In particular, $m(I)\in\bn^{3}$ generates the $\bq$-linear subspace of
solutions of system~(\ref{m-system1}) in Example~\ref{as*} (see
Remark~\ref{unique-sol}), and so $m_{i}(I)$ is the weight given to
$x_{i}$, $i=1,2,3$, in Example~\ref{as*}.  }\end{Remark}

\begin{proof}
We have $x_{1}A+I=(x_{1},x_{2}^{c_{2}},x_{2}^{b_{2}}x_{3}^{a_{3}},
x_{3}^{c_{3}})$, which is $\fm$-primary. Thus $m_{1}(I)$ is finite and
can be calculated as ${\rm
  length}(A_{\mathfrak{m}}/(x_{1}A+I)_{\mathfrak{m}})$. In particular,
$x_{1}\cdot A_{\mathfrak{m}}/I_{\mathfrak{m}}$ is a parameter ideal of
the Cohen-Macaulay local ring $A_{\mathfrak{m}}/I_{\mathfrak{m}}$
(recall that $I$ is height-unmixed; see
Proposition~\ref{ifgradev2}). By \cite[Proposition~11.1.10]{sh}, ${\rm
  length}(A_{\mathfrak{m}}/(x_{1}A+I)_{\mathfrak{m}})= e(x_{1}\cdot
A_{\mathfrak{m}}/I_{\mathfrak{m}})$. On the other hand, the quotient
ring $A/(x_{1}A+I)$ is isomorphic to
$k[x_{2},x_{3}]/(x_{2}^{c_{2}},x_{2}^{b_{2}}x_{3}^{a_{3}},x_{3}^{c_{3}})$
which has length $(a_{2}+b_{2})(a_{3}+b_{3})-a_{2}b_{3}$. There are
analogous arguments for $m_{2}(I)$ and $m_{3}(I)$.
\end{proof}

Now, let us extend the definition of Herzog ideals introduced in the
first section.

\begin{Definition}{\rm 
Let $n=(n_{1},n_{2},n_{3})\in\bn^{3}$ be an integer vector with
greatest common divisor not necessarily equal to 1. The {\em Herzog
  ideal associated to $n$} is the prime ideal $\fp_{n}$ defined as the
kernel of the morphism $\varphi_{n}:A\rightarrow k[t]$ sending $x_{i}$
to $t^{n_{i}}$ for each $i=1,2,3$.  }\end{Definition}

We then have the following:

\begin{Remark}\label{pm=pn}{\rm 
Let $m=(m_{1},m_{2},m_{3})\in\bn^{3}$ and
$n=(n_{1},n_{2},n_{3})\in\bn^{3}$ such that $m=dn$ for some $d\in\bn$.
Then $\fp_{m}=\fp_{n}$.  }\end{Remark}

\begin{proof}
Since $A/\fp_{n}\cong k[t^{n_{1}},t^{n_{2}},t^{n_{3}}]\subset k[t]$ is
an integral extension, ${\rm dim}(A/\fp_{n})=1$ and $\fp_{n}$ is a
prime ideal of height 2. Analogously, $\fp_{m}$ is a prime ideal of
height 2. If ${\rm gcd}(n_{1},n_{2},n_{3})=1$, using an explicit
system of generators of $\fp_{n}$ (given by Herzog in \cite{herzog};
see also \cite[pages~138-139]{kunz}), one can easily check that
$\fp_{n}\subseteq \fp_{m}$, and so they are equal. In general,
factoring out the greatest common divisor $e={\rm
  gcd}(n_{1},n_{2},n_{3})$, let $r=n/e=(r_{1},r_{2},r_{3})$, where
${\rm gcd}(r_{1},r_{2}.r_{3})=1$. Then $n=er$ and $m=(de)r$, thus
$\fp_{m}=\fp_{n}=\fp_{r}$.
\end{proof}

\begin{Lemma}\label{edepn}
Let $\fp_{n}$ be the Herzog ideal associated to
$n=(n_{1},n_{2},n_{3})\in\bn^{3}$. Then, for all $i=1,2,3$,
$x_{i}A+\fp_{n}$ is $\fm$-primary and $m_{i}(\fp_{n})=e(x_{i}\cdot
A_{\mathfrak{m}}/(\fp_{n})_{\mathfrak{m}})$. Moreover
$m(\fp_{n})=n/{\rm gcd}(n)$.
\end{Lemma}

\begin{proof}
By the preceding remark we clearly can suppose that ${\rm
  gcd}(n_{1},n_{2},n_{3})=1$. Since
$x_{2}^{n_{1}}=x_{1}^{n_{2}}+(x_{2}^{n_{1}}-x_{1}^{n_{2}})\in
x_{1}A+\fp_{n}$, and analogously for $x_{3}$, $x_{1}A+\fp_{n}$ is
$\fm$-primary. Thus $m_{1}(\fp_{n})$ is finite and can be calculated
as ${\rm length}
(A_{\mathfrak{m}}/(x_{1}A+\fp_{n})_{\mathfrak{m}})$. In particular,
$x_{1}\cdot A_{\mathfrak{m}}/(\fp_{n})_{\mathfrak{m}}$ is an $\fm
A_{\mathfrak{m}}/(\fp_{n})_{\mathfrak{m}}$-primary ideal of the
Cohen-Macaulay one-dimensional domain
$A_{\mathfrak{m}}/(\fp_{n})_{\mathfrak{m}}$. By
\cite[Proposition~11.1.10]{sh}, ${\rm length}
(A_{\mathfrak{m}}/(x_{1}A+\fp_{n})_{\mathfrak{m}})=e(x_{1}\cdot
A_{\mathfrak{m}}/(\fp_{n})_{\mathfrak{m}})$. On the other hand, the
quotient ring $A/(x_{1}A+\fp_{n})$ is isomorphic to $R/t^{n_{1}}R$,
where $R={\rm Im}(\varphi)=k[t^{n_{1}},t^{n_{2}},t^{n_{3}}]$, so ${\rm
  length}(A/(x_{1}A+\fp_{n}))={\rm length}(R/t^{n_{1}}R)$. To
calculate the latter one can localise at the maximal ideal $\fn
=(t^{n_{1}},t^{n_{2}},t^{n_{3}})$ since $t^{n_{1}}R$ is $\fn$-primary,
as $(t^{n_{j}})^{n_{1}}\in t^{n_{1}}R$, $j=2,3$. Since
$t^{n_{1}}R_{\mathfrak{n}}$ is a parameter ideal of a Cohen-Macaulay
local ring, ${\rm length}(R_{\mathfrak{n}}/t^{n_{1}}R_{\rm
  \mathfrak{n}})=e(t^{n_{1}}R_{\mathfrak{n}})$ (again by
\cite[Proposition~11.1.10]{sh}).

Because $1=s_{1}n_{1}+s_{2}n_{2}+s_{3}n_{3}$ for some $s_{i}\in\bz$,
$R\subset k[t]$ is a birational integral extension. Set
$S=R\setminus\fn$. Since $tk[t]$ is the only nonzero prime $\fq$ of $k[t]$
such that $\fq \cap S=\emptyset$, the saturation of $S$ in $k[t]$ is
$k[t]\setminus tk[t]$, so $R_{\mathfrak{n}}\subset k[t]_{(t)}$ is a
birational finite extension. Then
$e(t^{n_{1}}R_{\mathfrak{n}})=e(t^{n_{1}}k[t]_{(t)})$ (see
\cite[Corollary~11.2.6]{sh}). By \cite[Proposition~11.1.10]{sh} again,
the latter is equal to ${\rm
  length}(k[t]_{(t)}/t^{n_{1}}k[t]_{(t)})=n_{1}$.
\end{proof}

In particular, one has a kind of converse of Remark~\ref{pm=pn}.

\begin{Corollary}
Let $m=(m_{1},m_{2},m_{3})\in\bn^{3}$ and
$n=(n_{1},n_{2},n_{3})\in\bn^{3}$. Then $\fp_{m}=\fp_{n}$ if and only
if $m$ and $n$ are linearly dependent over the field $\bq$.
\end{Corollary}

\begin{proof}
If $\fp_{m}=\fp_{n}$, then $m(\fp_{m})=m(\fp_{n})$ and, by
Lemma~\ref{edepn}, ${\rm gcd}(n)\cdot m={\rm gcd}(m)\cdot
n$. Conversely, if $m$ and $n$ are linearly dependent over the field
$\bq$, then $rm=sn$ for some $r,s\in \bn$. By Remark~\ref{pm=pn},
$\fp_{m}=\fp_{n}$.
\end{proof}

Now, given an HN ideal $I$, we want to look for the ``nearest'' Herzog
ideal to $I$.

\begin{Lemma}\label{unique}
Let $I$ be an HN ideal and $m(I)\in\bn^{3}$ its associated integer
vector. Then $\fp_{m(I)}$ is the unique Herzog ideal containing
$I$. In particular, $\fp_{m(I)}$ is a minimal prime of $I$.
\end{Lemma}

\begin{proof} 
Given any $m=(m_{1},m_{2},m_{3})\in\bn^{3}$,
$I=(x_{1}^{c_{1}}-x_{2}^{b_{2}}x_{3}^{a_{3}}, x_{2}^{c_{2}}-
x_{1}^{a_{1}}x_{3}^{b_{3}},x_{3}^{c_{3}}-x_{1}^{b_{1}}x_{2}^{a_{2}})
\subseteq\fp_{m}$ if and only if $m$ satisfies
system~(\ref{m-system1}) of Example~\ref{as*}, whose $\bq$-linear
subspace of solutions is generated by $m(I)\in\bn^{3}$ (see
Example~\ref{as*} and Remark~\ref{edei}). Thus $I\subseteq
\fp_{m(I)}$. Since both ideals have height 2, $\fp_{m(I)}$ is a
minimal prime of $I$.

Suppose now that $I\subseteq \fp_{m}$ and $I\subseteq \fp_{n}$, where
$m,n\in\bn^{3}$. By Remark~\ref{pm=pn} we can suppose that ${\rm
  gcd}(m)=1$ and ${\rm gcd}(n)=1$. In particular, since $I\subseteq
\fp_{m},\fp_{n}$, then $m$ and $n$ are solutions of
system~(\ref{m-system1}) in Example~\ref{as*}. So, there exist
$p,q\in\bq$, $p,q>0$, such that $m=pm(I)$ and $n=qm(I)$, i.e., $rm=sn$
for some $r,s\in\bn$ (cf. Remark~\ref{unique-sol}). Taking the
greatest common divisor, $r=s$, $m=n$ and $\fp_{m}=\fp_{n}$.
\end{proof}

The next result characterizes when an HN ideal is a prime ideal.

\begin{Theorem}\label{theorem-prime}
Let $I$ be an HN ideal and $m(I)\in\bn^{3}$ its associated integer
vector. Then the following conditions are equivalent:
\begin{itemize}
\item[$(i)$] $I$ is prime;
\item[$(ii)$] $I=\fp_{m(I)}$;
\item[$(iii)$] ${\rm gcd}(m(I))=1$.
\end{itemize}
\end{Theorem}

\begin{proof}
By Lemma~\ref{unique}, $\fp_{m(I)}$ is a minimal prime of $I$. Thus
$I$ is prime if and only if $I=\fp_{m(I)}$. Consider the exact
sequence
\begin{eqnarray*}
0\rightarrow L\rightarrow A/I\rightarrow A/\fp_{m(I)}\rightarrow 0,
\end{eqnarray*}
where $L=\fp_{m(I)}/I$. Tensoring it with $A/x_{1}A$ and using that
$x_{1}\not\in\fp_{m(I)}$, one obtains
\begin{eqnarray*}
0\rightarrow L/x_{1}L\rightarrow A/(x_{1}A+I)\rightarrow
A/(x_{1}A+\fp_{m(I)})\rightarrow 0.
\end{eqnarray*}
Endow $A$ with the natural grading induced by giving $x_{i}$ weight
$m_{i}(I)$, $i=1,2,3$; see Remark~\ref{edei}. Then $I$ and
$\fp_{m(I)}$ are homogeneous ideals in this grading. By the graded
variant of Nakayama's lemma, $L=0$ if and only if
$L=x_{1}L$. Therefore, $I=\fp_{m(I)}$ if and only if the Artinian
rings $A/(x_{1}A+I)$ and $A/(x_{1}A+\fp_{m(I)})$ have the same
length. But, by Lemma~\ref{edepn}, the length of
$A/(x_{1}A+\fp_{m(I)})$ is equal to $m_{1}(\fp_{m(I)})=m_{1}(I)/{\rm
  gcd}(m(I))$, where we recall that $m_{1}(I)$ is by definition the
length of $A/(x_{1}A+I)$. So the result follows.
\end{proof}

\begin{Example}\label{ex-ra-rb}
{\rm Let $I_{r}$ be the HN ideal associated to
\begin{eqnarray*}
  \cm_{r}=\left(\begin{array}{ccc}x_{1}^{ra_{1}}&x_{2}^{ra_{2}}&x_{3}^{ra_{3}}
    \\ x_{2}^{rb_{2}}&x_{3}^{rb_{3}}&x_{1}^{rb_{1}}\end{array}\right), 
\end{eqnarray*}
where $a,b\in\bn^{3}$ and $r\in\bn$. Then $m(I_{r})=r^{2}m(I_{1})$. In
particular, $I_{r}$ is not prime for all $r>1$. In fact,
$I_{r}\subseteq I_{1}$, since
$x_{1}^{rc_{1}}-x_{2}^{rb_{2}}x_{3}^{ra_{3}}=
(x_{1}^{c_{1}}-x_{2}^{b_{2}}x_{3}^{a_{3}})
(\sum_{i=1}^{r}x_{1}^{(r-i)c_{1}}x_{2}^{(i-1)b_{2}}x_{3}^{(i-1)a_{3}})$. Note
that the $I_{r}$ are all distinct and all have the same associated
Herzog ideal. }\end{Example}

\begin{Remark}{\rm 
One has the maps $m:\{ \mbox{HN ideals}\} \rightarrow \bn^{3}$ and
$\fp_{\bullet}:\bn^{3}\rightarrow \{\mbox{Herzog ideals}\}$ defined by
$I\mapsto m(I)$ and $m\mapsto \fp_{m}$.

The map $m$ is not injective. For example, as regards the HN ideal
given by the triples $(a_{1},a_{2},a_{3})=(1,1,3)$,
$(b_{1},b_{2},b_{3})=(3,2,3)$, we get $m=(15,15,10)$, which we also
get whenever $(a_{1},a_{2},a_{3})=(2,3,3)$,
$(b_{1},b_{2},b_{3})=(1,1,3)$. The second ideal contains a binomial
with a pure term in $x_{1}^{3}$, whereas the first ideal does not.
And indeed $m$ is not surjective even if we restrict the range to
triples of positive integers each of which is at least 3; for example
$(3,4,4)$ is not in the image of $m$, since if it were,
$a_{2}=a_{3}=b_{2}=b_{3}=1$, and a contradiction would follow easily.
The map $\fp_{\bullet}$ is not injective by Remark~\ref{pm=pn} and is
surjective by definition.

The composition $\fp_{\bullet}\circ m$, which assigns to each HN ideal
its associated Herzog ideal, is clearly not injective (because $m$ is
not injective; see also Example~\ref{ex-ra-rb}). 

Is $\fp_{\bullet}\circ m$ surjective? What is the image of $m$?
}\end{Remark}

\begin{Remark}{\rm 
As regards a description of ${\rm Im}(\fp_{\bullet}\circ m)$ and ${\rm
  Im}(m)$, at the moment we have only partial results giving necessary
conditions for triples of positive integers to belong to these image
sets. These results have elementary but somewhat lengthy and technical
proofs. So for the moment we confine ourselves to the following
observations.

(1) Whenever $n=(n_{1},n_{2},n_{3})\in\bn^{3}$ has ${\rm gcd}(n)=1$,
then $n\in {\rm Im}(m)$ if and only if the subsemigroup $H$ of $(\bn
,+)$ generated by $n_{1},n_{2},n_{3}$ is not symmetric. Indeed,
suppose that there exists an HN ideal $I$ with $m(I)=n$. By
Theorem~\ref{theorem-prime}, $I$ is prime and equal to $\fp_{n}$. Thus
$\fp_{n}$ is an HN ideal and so is not a complete intersection. By
\cite{herzog} (see also \cite[p.~139]{kunz}), $H$ is not
symmetric. Conversely, if $H$ is not symmetric, $\fp_{n}$ is not a
complete intersection. Thus $\fp_{n}$ is an HN ideal and, by
Remark~\ref{edepn}, $m(\fp_{n})=n$, so $n\in {\rm Im}(m)$.

(2) If $(m_{1},m_{2},m_{3})$ is in ${\rm Im}(m)$ and $m_{1}$ and
$m_{2}$ are bounded above by $r$, say, then it is easy to see that
$m_{3}$ is bounded above by $3r^{2} - 2$. On the other hand, one can
show that if $n=(n_{1},n_{2},n_{3})\in\bn^{3}$ is such that ${\rm
  gcd}(n_{2},n_{3})= 1$, so that in particular ${\rm
  gcd}(n_{1},n_{2},n_{3}) = 1$, and that $n_{1}$ is contained in $\bn
n_{2}+\bn n_{3}$, then $\fp_{n}$ lies in ${\rm Im}(\fp_{\bullet}\circ
m)$. }\end{Remark}

\section{On the number of primary components of an HN ideal
}\label{number}

We keep the notations of the former section, i.e.,
$A=k[x_{1},x_{2},x_{3}]$ is the polynomial ring in three variables
$x=x_{1},x_{2},x_{3}$ over a field $k$ and $\fm$ is the maximal ideal
of $A$ generated by $x_{1},x_{2},x_{3}$. The purpose of this section
is to give bounds for the number of associated (i.e., minimal) primes
of an HN ideal of $A$. We begin with the following observation.

\begin{Remark}\label{af}{\rm 
Let $I$ be an HN ideal and $m(I)$ its associated integer vector. Then,
for each $i=1,2,3$,
\begin{eqnarray*}
m_{i}(I)=\sum_{\fq\in\mbox{Min}(A/I)} e(x_{i}\cdot
A_{\mathfrak{m}}/\fq_{\mathfrak{m}})\mbox{length}((A/I)_{\fq}).
\end{eqnarray*}
In particular, $\mbox{cardinal}\, {\rm Min}(A/I)\leq {\rm min}\{
m_{1}(I),m_{2}(I),m_{3}(I)\}$.  }\end{Remark}

\begin{proof}
By Remark~\ref{edei}, $m_{1}(I)=e(x_{1}\cdot
A_{\mathfrak{m}}/I_{\mathfrak{m}})$ and, by the Associativity Formula
(see \cite[Theorem~11.2.4]{sh}), this can be calculated as
\begin{eqnarray*}
e(x_{1}\cdot A_{\mathfrak{m}}/I_{\mathfrak{m}})=
\sum_{\fa\in\mbox{Min}(A_{\mathfrak{m}}/I_{\mathfrak{m}})}
e(x_{1}\cdot A_{\mathfrak{m}}/I_{\mathfrak{m}}; A_{\mathfrak{m}}/\fa)
\mbox{length}((A_{\mathfrak{m}}/I_{\mathfrak{m}})_{\fa}).
\end{eqnarray*}
Endowing $A$ with the grading obtained by giving $x_{i}$ weight
$m_{i}(I)$, $i=1,2,3$, $I$ is then homogeneous in this grading and
hence any associated prime $\fq$ of $I$ sits inside $\fm$. Thus any
$\fa$ in ${\rm Min}(A_{\mathfrak{m}}/I_{\mathfrak{m}})$ is of the form
$\fa=\fq A_{\mathfrak{m}}$ with $\fq$ in ${\rm Min}(A/I)$ and vice
versa. Hence the equality follows. In particular, the sum has as many
non-zero terms as $I$ has minimal primes.
\end{proof}

To improve on this observation, we need the following lemma which was
inspired by, and in turn generalizes, \cite[Lemma~10.15]{fossum}.

\begin{Lemma}\label{mainlemma}
Let $m_{2},m_{3}\in\bn$ with ${\rm gcd}(m_{2},m_{3})=e$. Let
$m_{2}=ep_{2}$ and $m_{3}=ep_{3}$ with ${\rm gcd}(p_{2},p_{3})=1$. Let
$f$ be a factor of $x_{2}^{m_{3}}-x_{3}^{m_{2}}$ which is not a
unit. Then $f$ is of the form
\begin{eqnarray*}
a_{rp_{3}}x_{2}^{rp_{3}}+a_{(r-1)p_{3}}x_{2}^{(r-1)p_{3}}x_{3}^{p_{2}}+\ldots
+a_{p_{3}}x_{2}^{p_{3}}x_{3}^{(r-1)p_{2}}+a_{0}x_{3}^{rp_{2}},
\end{eqnarray*}
with $r\in\bn$ and $a_{i}\in k$, $a_{rp_{3}},a_{0}\neq 0$.
\end{Lemma}

\begin{proof} 
Set the weight of $x_{i}$ equal to $p_{i}$, $i=2,3$, so that
$x_{2}^{m_{3}}-x_{3}^{m_{2}}$ is homogeneous of degree
$ep_{2}p_{3}$. Suppose that $f$ is a factor of
$x_{2}^{m_{3}}-x_{3}^{m_{2}}\in k[x_{2},x_{3}]$ which is not a
unit. Then $f$ is homogeneous of degree $p$, say, where $p>0$. Write
\begin{eqnarray*}
f=a_{t}x_{2}^{t}+a_{t-1}x_{2}^{t-1}x_{3}^{l_{t-1}}+\ldots
+a_{0}x_{3}^{l_{0}},
\end{eqnarray*}
$a_{t},a_{0}\neq 0$, with typical term $a_{i}x_{2}^{i}x_{3}^{l_{i}}$,
for $0\leq i\leq t$. So $ip_{2}+l_{i}p_{3}=p=tp_{2}$. Because ${\rm
  gcd}(p_{2},p_{3})=1$, $p_{2}$ divides $l_{i}$ and $p_{3}$ divides
$t-i$. Since $tp_{2}=p<p+q=(e-1)p_{2}p_{3}$, then $0\leq i\leq
t<(e-1)p_{3}$. So $0\leq t-i<(e-1)p_{3}$. Thus $t-i=jp_{3}$, and hence
$i=t-jp_{3}$, with $j\in\bz$, $0\leq j\leq e-2$. Therefore
\begin{eqnarray*}
f=a_{t}x_{2}^{t}+a_{t-p_{3}}x_{2}^{t-p_{3}}x_{3}^{l_{t-p_{3}}}+
a_{t-2p_{3}}x_{2}^{t-2p_{3}}x_{3}^{l_{t-2p_{3}}}+\ldots +
a_{t-(e-2)p_{3}}x_{2}^{t-(e-2)p_{3}}x_{3}^{l_{t-(e-2)p_{3}}}.
\end{eqnarray*}
Since $f$ has a pure term in $x_{3}$, then $t=rp_{3}$ for some
$r\in\bn$, $1\leq r\leq e-2$, and
\begin{eqnarray*}
f=a_{rp_{3}}x_{2}^{rp_{3}}+
a_{(r-1)p_{3}}x_{2}^{(r-1)p_{3}}x_{3}^{l_{(r-1)p_{3}}}+\ldots
+a_{p_{3}}x_{2}^{p_{3}}x_{3}^{l_{p_{3}}}+a_{0}x_{3}^{l_{0}}.
\end{eqnarray*}
For $0\leq j\leq r$, the term with coefficient $a_{(r-j)p_{3}}$ has
degree $(r-j)p_{3}p_{2}+l_{(r-j)p_{3}}p_{3}$, which must be equal to
$rp_{2}p_{3}$, the degree of $f$. Thus
$l_{(r-j)p_{3}}=jp_{2}$. Therefore
\begin{eqnarray*}
f=a_{rp_{3}}x_{2}^{rp_{3}}+
a_{(r-1)p_{3}}x_{2}^{(r-1)p_{3}}x_{3}^{p_{2}}+\ldots +
a_{p_{3}}x_{2}^{p_{3}}x_{3}^{(r-1)p_{2}}+a_{0}x_{3}^{rp_{2}},
\end{eqnarray*}
as desired.
\end{proof}

\begin{Theorem}\label{bound}
Let $I$ be an HN ideal and $m(I)$ its associated integer vector. Let
$\fq\in {\rm Min}(A/I)$. Then, for all $i,j\in \{1,2,3\}$, $i\neq j$,
\begin{eqnarray*}
e(x_{i};A_{\mathfrak{m}}/\fq_{\mathfrak{m}})\geq m_{i}(I)/{\rm
  gcd}(m_{i}(I),m_{j}(I)).
\end{eqnarray*}
In particular, setting $d={\rm gcd}(m(I))$ and $s={\rm min}\{ {\rm
  gcd}(m_{i}(I),m_{j}(I))\mid 1\leq i<j\leq 3\}$,
\begin{eqnarray*}
{\rm cardinal}\, {\rm Min}(A/I)\leq 1+(s\cdot (d-1)/d).
\end{eqnarray*}
\end{Theorem}

\begin{proof}
Take $\fq$ a minimal prime of $I$ in $A$, and set $B=k[x_{2},x_{3}]$
and $\fn=(x_{2},x_{3})$ the maximal ideal of $B$ generated by
$x_{2},x_{3}$. Clearly $\fq\cap B\subset \fn$. Note that $A/\fq$ is a
finite $(B/\fq\cap B)$-module. In particular $(A/\fq)_{\mathfrak{n}}$
is a finite $(B/\fq\cap B)_{\mathfrak{n}}$-module. But $\fm$ is the
unique maximal (indeed prime) ideal $\fp\supseteq \fq$ in $A$ such
that $\fp/\fq \cap (B/\fq\cap B)=\fn/\fq\cap B$ (since in $A/\fq$,
$x_{1}^{c_{1}}=x_{2}^{b_{2}}x_{3}^{a_{3}}$), so that
$(A/\fq)_{\mathfrak{n}}=(A/\fq)_{\mathfrak{m}}$. Therefore,
$(A/\fq)_{\mathfrak{m}}$ is a finite $(B/\fq\cap
B)_{\mathfrak{n}}$-module.

Moreover, in $A/I$, $x_{2}^{c_{2}}=x_{1}^{a_{1}}x_{3}^{b_{3}}$ and
$x_{3}^{c_{3}}=x_{1}^{b_{1}}x_{2}^{a_{2}}$. Write
$m(I)=(m_{1},m_{2},m_{3})$. Since $m_{2}=a_{1}c_{3}+b_{1}b_{3}$ and
$m_{3}=a_{1}a_{2}+b_{1}c_{2}$, then
$x_{2}^{m_{3}}=x_{2}^{a_{1}a_{2}}x_{2}^{b_{1}c_{2}}=
x_{1}^{a_{1}b_{1}}x_{2}^{a_{1}a_{2}}x_{3}^{b_{1}b_{3}}$ and
$x_{3}^{m_{2}}=x_{3}^{a_{1}c_{3}}x_{3}^{b_{1}b_{3}}=
x_{1}^{a_{1}b_{1}}x_{2}^{a_{1}a_{2}}x_{3}^{b_{1}b_{3}}$. Thus
$x_{2}^{m_{3}}=x_{3}^{m_{2}}$ in $A/I$ and
$x_{2}^{m_{3}}-x_{3}^{m_{2}}\in I$. In particular,
$x_{2}^{m_{3}}-x_{3}^{m_{2}}\in I\cap B\subseteq \fq\cap B$. Thus
$(x_{2},x_{3}^{m_{2}})\subseteq x_{2}B+(\fq\cap B)$ and $x_{2}\cdot
B_{\mathfrak{n}}/(\fq\cap B)_{\mathfrak{n}}$ is
$\mathfrak{n}B_{\mathfrak{n}}/(\fq\cap B)_{\mathfrak{n}}$-primary.
Therefore, by \cite[Corollary~11.2.6]{sh},
\begin{eqnarray*}
e(x_{2};A_{\mathfrak{m}}/\fq_{\mathfrak{m}})=e(x_{2};B_{\mathfrak{n}}/(\fq\cap
B)_{\mathfrak{n}})\cdot {\rm rank}_{B_{\mathfrak{n}}/(\fq\cap
  B)_{\mathfrak{n}}}(A_{\mathfrak{m}}/\fq_{\mathfrak{m}})\geq
e(x_{2};B_{\mathfrak{n}}/(\fq\cap B)_{\mathfrak{n}}).
\end{eqnarray*}
Since $B/\fq\cap B\hookrightarrow A/\fq$ is an integral extension,
then $\fq\cap B$ is a prime ideal of height 1 of $B$ (thus principal)
and $B_{\mathfrak{n}}/(\fq\cap B)_{\mathfrak{n}}$ is a one-dimensional
Noetherian local domain, thus Cohen-Macaulay. By
\cite[Proposition~11.1.10]{sh}, $e(x_{2};B_{\mathfrak{n}}/(\fq\cap
B)_{\mathfrak{n}})={\rm length}
(B_{\mathfrak{n}}/x_{2}B_{\mathfrak{n}}+(\fq\cap B)_{\mathfrak{n}})$.
Since $x_{2}B+\fq\cap B$ is $\fn$-primary, the latter length is equal
to ${\rm length} (B/x_{2}B+\fq\cap B)$.

But $\fq\cap B=(f)$ for some irreducible polynomial $f\in B$ and
$x_{2}^{m_{3}}-x_{3}^{m_{2}}\in \fq\cap B=(f)$. Thus
$x_{2}^{m_{3}}-x_{3}^{m_{2}}=fg$ for some $g\in B$. By
Lemma~\ref{mainlemma}, and following its notations, $f$ is a
polynomial in $x_{2}^{p_{3}}$ and $x_{3}^{p_{2}}$ with pure non-zero
terms in each of $x_{2}^{p_{3}}$ and $x_{3}^{p_{2}}$, where ${\rm
  gcd}(m_{2},m_{3})=e$ and $m_{2}=ep_{2}$ and $m_{3}=ep_{3}$, with
${\rm gcd}(p_{2},p_{3})=1$. Thus $\fq\cap B\subseteq
(x_{2},x_{3}^{p_{2}})$. Hence
$e(x_{2};A_{\mathfrak{m}}/\fq_{\mathfrak{m}}) \geq {\rm
  length}(B/x_{2}B+\fq\cap B)\geq {\rm
  length}(B/(x_{2},x_{3}^{p_{2}}))=p_{2}=m_{2}/e$.

By Remark~\ref{af} and Lemma~\ref{edepn},
\begin{eqnarray*}
m_{2}&&\geq e(x_{2};A_{\mathfrak{m}}/(\fp_{m(I)})_{\mathfrak{m}})+
\sum_{\mathfrak{q}\in\mbox{Min}(A/I)\setminus \{ \mathfrak{p}_{m(I)}
  \}} e(x_{2};A_{\mathfrak{m}}/\fq_{\mathfrak{m}}) \\ &&\geq
(m_{2}/d)+(m_{2}/e)\cdot ({\rm cardinal}\, {\rm Min}(A/I)-1).
\end{eqnarray*}
So ${\rm cardinal}\, {\rm Min}(A/I)\leq 1+(e\cdot (d-1)/d)$.
\end{proof}

\begin{Example}{\rm 
Let $I=(x_{1}^{15}-x_{2}^{8}x_{3}^{3}, x_{2}^{10}-x_{1}^{5}x_{3}^{6},
x_{3}^{9}-x_{1}^{10}x_{2}^{2})$ be the HN ideal associated to
\begin{eqnarray*}
\cm_{r}=\left(\begin{array}{ccc}x_{1}^{5}&x_{2}^{2}&x_{3}^{3}
\\ x_{2}^{8}&x_{3}^{6}&x_{1}^{10}\end{array}\right). 
\end{eqnarray*}
Then $m(I)=(78,105,110)$, whose greatest common divisor is 1, thus
$I=\fp_{m(I)}$ is prime by Theorem~\ref{theorem-prime}. Note that for
HN ideals which are prime, the bound given in Theorem~\ref{bound} is
precisely 1. However it may happen that $s={\rm min}\{ {\rm
  gcd}(m_{i}(I),m_{j}(I))\mid 1\leq i<j\leq 3\}\neq 1$. For instance,
in this case, $s=2$.}\end{Example}

\begin{Example}\label{char2}{\rm
Let $I=(x_{1}^{5}-x_{2}^{2}x_{3},x_{2}^{4}-x_{1}x_{3}^{3},
x_{3}^{4}-x_{1}^{4}x_{2}^{2})$ be the HN ideal associated to
\begin{eqnarray*}
\cm=\left(\begin{array}{ccc}x_{1}^{}&x_{2}^{2}&x_{3}^{}
  \\ x_{2}^{2}&x_{3}^{3}&x_{1}^{4}\end{array}\right).
\end{eqnarray*}
Then $m(I)=(10,16,18)$, whose greatest common divisor $d$ equals 2.
Therefore, by Theorem~\ref{theorem-prime}, $I$ is not prime. Moreover,
${\rm min}\{ {\rm gcd}(m_{i}(I),m_{j}(I))\mid 1\leq i<j\leq
3\}=2$. Thus, by Theorem~\ref{bound}, $I$ has at most two minimal
primes. By Lemma~\ref{unique}, the Herzog ideal $\fp_{n}$ associated
to $n=m(I)/d=(5,8,9)$ is a minimal prime of $I$. To calculate
$\fp_{n}$, we use \cite[pages~137-139]{kunz}: since $5$ is the least
integer number $c_{1}$ such that $c_{1}n_{1}\in\bn n_{2}+\bn n_{3}$
($5\cdot 5=2\cdot 8+1\cdot 9$), $3$ is the least integer number
$c_{2}$ such that $c_{2}n_{2}\in\bn n_{1}+\bn n_{3}$ ($3\cdot 8=3\cdot
5+1\cdot 9$) and $2$ is the least integer number $c_{3}$ such that
$c_{3}n_{3}\in\bn n_{1}+\bn n_{2}$ ($2\cdot 9=2\cdot 5+1\cdot 8$),
then
$\fp_{n}=(x_{1}^{5}-x_{2}^{2}x_{3}^{},x_{2}^{3}-x_{1}^{3}x_{3}^{},
x_{3}^{2}-x_{1}^{2}x_{2}^{})$. Observe that $\fp_{n}$ is the HN ideal
associated to the matrix
\begin{eqnarray*}
\cm_{1}=\left(\begin{array}{ccc}x_{1}^{3}&x_{2}^{}&x_{3}^{}
\\ x_{2}^{2}&x_{3}^{}&x_{1}^{2}\end{array}\right).
\end{eqnarray*}
Moreover, if ${\rm char}(k)\neq 2$, the $k$-algebra automorphism
$\psi:A\rightarrow A$ defined by $\psi(x_{1})=x_{1}$,
$\psi(x_{2})=-x_{2}$ and $\psi(x_{3})=x_{3}$, leaves $I$ invariant
whereas it takes $\fp_{n}$ to the prime ideal
$\fq=(x_{1}^{5}-x_{2}^{2}x_{3}^{},x_{2}^{3}+x_{1}^{3}x_{3}^{},
x_{3}^{2}+x_{1}^{2}x_{2}^{})$. In other words, $I=\psi(I)\subset \psi
(\fp_{n})=\fq$, and $\fq$ is also a minimal prime of $I$. Thus the
bound in Theorem~\ref{bound} is attained.  In fact,
$e(x_{1};A_{\mathfrak{m}}/\fa_{\mathfrak{m}})\geq m_{1}(I)/{\rm
  gcd}(m_{1}(I),m_{2}(I))=5$ for any minimal prime $\fa$ of $I$. Thus,
by Remark~\ref{af}, ${\rm length}((A/I)_{\fa})=1$ for each such $\fa$,
and $I=\fp_{n}\cap \fq$ is radical. We will see in the next section
that this fact holds more generally.

Note that if ${\rm char}(k)=2$, then $ (x_{2}^{3}-x_{1}^{3}x_{3})^{2}=
-x_{1}x_{3}^{2}(x_{1}^{5}-x_{2}^{2}x_{3})+
x_{2}^{2}(x_{2}^{4}-x_{1}x_{3}^{3})$ and
$(x_{3}^{2}-x_{1}^{2}x_{2})^{2}=x_{3}^{4}-x_{1}^{4}x_{2}^{2}$.
Therefore $\fp_{n}^{2}\subseteq I\subsetneq \fp_{n}$ and ${\rm
  rad}(I)=\fp_{n}$. In particular $I$ has only one minimal prime and
is not radical (because it is not prime). Note that
$\fp_{n}^{2}\subsetneq I$. }\end{Example}

\begin{Remark}{\rm 
Let $I$ be an HN ideal and $m(I)=m=(m_{1},m_{2},m_{3})\in\bn^{3}$ its
associated integer vector. Let $d={\rm gcd}(m)$ and
$n=m/d=(n_{1},n_{2},n_{3})\in\bn^{3}$, where ${\rm gcd}(n)=1$. Then
${\rm gcd}(m_{i},m_{j})=d\cdot {\rm gcd}(n_{i},n_{j})$ for all $1\leq
i<j\leq 3$. Thus
\begin{eqnarray*}
s={\rm min}\{ {\rm gcd}(m_{i},m_{j})\mid 1\leq i<j\leq 3\}=d\cdot {\rm
  min}\{ {\rm gcd}(n_{i},n_{j})\mid 1\leq i<j\leq 3\}=d\cdot r,
\end{eqnarray*}
say. Therefore $1+(s(d-1)/d)=1+r(d-1)$. In particular, if some ${\rm
  gcd}(n_{i},n_{j})=1$, then ${\rm cardinal}\, {\rm Min}(A/I)\leq
d$. Does this bound hold in complete generality?  
}\end{Remark}

\begin{Remark}{\rm 
We have established two other estimates for the number of minimal
primes in $A/I$.  In a number of cases that we have looked at, these
estimates are weaker than the one given in the statement of
Theorem~\ref{bound}.  However, the methods used to obtain them are of
interest and the estimates themselves may prove to be of worth in
other situations.  So we confine ourselves to sketching some brief
details concerning them.

\noindent $(1)$ By the comment right at the end of Section~\ref{licci},
the number of minimal primes in $A/I$ is one less than the number of
minimal primes in $A/(v)$.  Set $n = m(I)/{\rm gcd}(m(I))$.  We can
then use the argument of the proof of Remark~\ref{af}, only this time
applied to $e(x_{i}\cdot A_{\mathfrak{m}}/(v)A_{\mathfrak{m}})$, $i =
1,2,3$, to get the following estimate:
\begin{eqnarray*}
{\rm cardinal}\, {\rm Min}(A/I)\leq {\rm
  min}\{c_{i}c_{j}-n_{k}\mid\{i,j,k\}= \{1,2,3\}\}.
\end{eqnarray*}
Here we have used symmetry and the fact that $e(x_{i}\cdot
A_{\mathfrak{m}}/(\fp_{n})_{\mathfrak{m}})=n_{i}$ (see
Lemma~\ref{edepn}).

\noindent $(2)$ A more delicate argument using minimal reductions and
the criterion of multiplicity 1 establishes the following result.
Suppose, possibly after relabelling the suffices, that $a_{3}\geq
c_{1}$, and that $b_{3}\geq c_{2}$.  Then
\begin{eqnarray*}
{\rm cardinal}\, {\rm Min}(A/I)\leq (c_{1}c_{2}-n_{3})/2.
\end{eqnarray*}
}\end{Remark}

\section{HN ideals are usually radical}\label{radical}

As in the previous section, $A=k[x_{1},x_{2},x_{3}]$ is the polynomial
ring with three variables over a field $k$. We start with the
following result.

\begin{Theorem}\label{theorem-radical}
Let $I$ be an HN ideal. If $k$ has characteristic zero (or large
enough), then ${\rm rad}(I)=I$.
\end{Theorem}

\begin{proof}
By Example~\ref{as*}, $v_{1},v_{2}$ is a regular sequence. Thus ${\rm
  rad}(v)=(v):{\rm Jac}(v)$ where ${\rm Jac}(v)$ is the Jacobian ideal
of $(v)$, i.e., the ideal generated by the $2\times 2$ minors of the
Jacobian matrix $\partial (v_{1},v_{2})/\partial (x_{1},x_{2},x_{3})$,
provided that $k$ has characteristic zero or sufficiently large (see
\cite[Theorem~5.4.2, page~131 and comments on
  page~130]{vasconcelos}). Concretely, setting
\begin{eqnarray*}
&&J_{1}=c_{1}c_{2}x_{1}^{c_{1}-1}x_{2}^{c_{2}-1}-
a_{1}b_{2}x_{1}^{a_{1}-1}x_{2}^{b_{2}-1}x_{3}^{c_{3}},\\ 
&&J_{2}=b_{3}c_{1}x_{1}^{a_{1}+c_{1}-1}x_{3}^{b_{3}-1}+
a_{1}a_{3}x_{1}^{a_{1}-1}x_{2}^{b_{2}}x_{3}^{c_{3}-1},\\ 
&&J_{3}=b_{2}b_{3}x_{1}^{a_{1}}x_{2}^{b_{2}-1}x_{3}^{c_{3}-1}+
a_{3}c_{2}x_{2}^{b_{2}+c_{2}-1}x_{3}^{a_{3}-1},
\end{eqnarray*}
these being the three generators of ${\rm Jac}(v)$,
\begin{eqnarray*}
{\rm rad}(v)=(v):{\rm Jac}(v)=(v):(J_{1},J_{2},J_{3})=
[(v):J_{1}]\cap [(v):J_{2}]\cap [(v):J_{3}]\subseteq (v):J_{1}.
\end{eqnarray*}
Write $J_{1}=x_{1}^{a_{1}-1}x_{2}^{b_{2}-1}h$, with
$h=-a_{1}b_{2}D+sx_{1}^{b_{1}}x_{2}^{a_{2}}\in A$,
$D=x_{3}^{c_{3}}-x_{1}^{b_{1}}x_{2}^{a_{2}}$ and
$s=c_{1}c_{2}-a_{1}b_{2}\in \bz$. Now, by Proposition~\ref{ifgradev2}
and Corollary~\ref{Icapu=v}, and using the general rule of quotient ideals
that $L:fg=(L:f):g$, we have
\begin{eqnarray*}
&&(v):J_{1}=((v):x_{1}^{a_{1}-1}x_{2}^{b_{2}-1}):h=[(I\cap
    (u)):x_{1}^{a_{1}-1}x_{2}^{b_{2}-1}]:h=
  \\&&[(I:x_{1}^{a_{1}-1}x_{2}^{b_{2}-1})\cap
    ((u):x_{1}^{a_{1}-1}x_{2}^{b_{2}-1})]:h=(I\cap (x_{1},x_{2})):h=
  \\ &&(I:h)\cap ((x_{1},x_{2}):h)=(I:x_{1}^{b_{1}}x_{2}^{a_{2}})\cap
  ((x_{1},x_{2}):x_{3}^{c_{3}})=I\cap (x_{1},x_{2})\subseteq I.
  \end{eqnarray*}
Therefore ${\rm rad}(v)\subseteq I$. By Lemma~\ref{ivu1}, ${\rm
  rad}(I)={\rm rad}((v):u_{1})\subseteq {\rm rad}(v):u_{1}\subseteq
I:u_{1}=I$.
\end{proof}

\begin{Remark}{\rm 
In particular, ${\rm rad}(v)={\rm rad}(I)\cap {\rm rad}(u)=I\cap
(x_{1},x_{2})=(v_{1},v_{2},x_{1}D,x_{2}D)$.  }\end{Remark}

\begin{Example}{\rm 
Let $I$ be the HN ideal considered in Example~\ref{char2}. By
Theorem~\ref{theorem-radical}, if $k$ has characteristic zero, $I$ is
radical. In fact, we have shown that, if ${\rm char}(k)\neq 2$, then
$I$ is radical, whereas if ${\rm char}(k)=2$, then $I$ is not radical.
}\end{Example}

This fact holds rather more generally (see the discussion in
Example~\ref{again} below). Before coming to this, we need the
following two results.

\begin{Proposition}\label{rv}
Let $I$ be an HN ideal and $m(I)=(m_{1},m_{2},m_{3})\in\bn^{3}$ its
associated integer vector. If $a_{1}=1$, then $(A/I)_{x_{2}x_{3}}$ is
isomorphic to
$(k[x_{2},x_{3}]/(x_{2}^{m_{3}}-x_{3}^{m_{2}}))_{x_{2}x_{3}}$. In
particular, their total quotient rings $Q(A/I)$ and
$Q(k[x_{2},x_{3}]/(x_{2}^{m_{3}}-x_{3}^{m_{2}}))$ are isomorphic.
Furthermore, the cardinality of ${\rm Min}(A/I)$ is equal to the
cardinality of a maximal complete set of orthogonal idempotents of
$Q(k[x_{2},x_{3}]/(x_{2}^{m_{3}}-x_{3}^{m_{2}}))$.
\end{Proposition}

\begin{proof}
One has $I=(v_{1},v_{2},D)$, with
$v_{1}=x_{1}^{c_{1}}-x_{2}^{b_{2}}x_{3}^{a_{3}}$,
$v_{2}=x_{2}^{c_{2}}-x_{1}^{a_{1}}x_{3}^{b_{3}}$,
$D=x_{3}^{c_{3}}-x_{1}^{b_{1}}x_{2}^{a_{2}}$. In $A_{x_{2}x_{3}}$,
since $a_{1}=1$,
$x_{1}=x_{2}^{c_{2}}x_{3}^{-b_{3}}-v_{2}x_{3}^{-b_{3}}$. Thus, in
$A_{x_{2}x_{3}}$,
\begin{eqnarray*}
v_{1}&&=(x_{2}^{c_{2}}x_{3}^{-b_{3}}-v_{2}x_{3}^{-b_{3}})^{c_{1}}
-x_{2}^{b_{2}}x_{3}^{a_{3}}=x_{2}^{c_{1}c_{2}}x_{3}^{-b_{3}c_{1}}
-x_{2}^{b_{2}}x_{3}^{a_{3}}+v_{2}p=\\ &&=x_{2}^{b_{2}}x_{3}^{-b_{3}c_{1}}
(x_{2}^{c_{1}c_{2}-b_{2}}-x_{3}^{a_{3}+b_{3}c_{1}})+v_{2}p=
x_{2}^{b_{2}}x_{3}^{-b_{3}c_{1}}(x_{2}^{m_{3}}-x_{3}^{m_{2}})+v_{2}p,\;
\mbox{ and } \\ D&&=x_{3}^{c_{3}}-
(x_{2}^{c_{2}}x_{3}^{-b_{3}}-v_{2}x_{3}^{-b_{3}})^{b_{1}}x_{2}^{a_{2}}=
x_{3}^{c_{3}}-x_{2}^{a_{2}+b_{1}c_{2}}x_{3}^{-b_{1}b_{3}}+v_{2}q=\\&&
=x_{3}^{-b_{1}b_{3}}(x_{3}^{b_{1}b_{3}+c_{3}}-x_{2}^{a_{2}+b_{1}c_{2}})+v_{2}q=
x_{3}^{-b_{1}b_{3}}(x_{3}^{m_{2}}-x_{2}^{m_{3}})+v_{2}q,
\end{eqnarray*}
with $p,q\in A_{x_{2}x_{3}}$. Therefore
$IA_{x_{2}x_{3}}=(x_{2}^{m_{3}}-x_{3}^{m_{2}},v_{2})A_{x_{2}x_{3}}=
(x_{2}^{m_{3}}-x_{3}^{m_{2}},x_{1}-x_{2}^{c_{2}}x_{3}^{-b_{3}})A_{x_{2}x_{3}}$.
So $(A/I)_{x_{2}x_{3}}\cong
(k[x_{2},x_{3}]/(x_{2}^{m_{3}}-x_{3}^{m_{2}}))_{x_{2}x_{3}}$. Since
$x_{2},x_{3}\in A$ are regular modulo $I$ and each of $x_{2},x_{3}\in
k[x_{2},x_{3}]$ is regular modulo $(x_{2}^{m_{3}}-x_{3}^{m_{2}})$,
the total quotient ring of $A/I$ is isomorphic to the total quotient
ring of $k[x_{2},x_{3}]/(x_{2}^{m_{3}}-x_{3}^{m_{2}})$.

Let $V$ be the affine $k$-variety defined by $I$ and $R(V)$ the ring
of rational functions of $V$, i.e., the total quotient ring $Q(A/I)$
of $A/I$ (see e.g. \cite[III, Proposition~3.4]{kunz}). So $R(V)\cong
Q(k[x_{2},x_{3}]/(x_{2}^{m_{3}}-x_{3}^{m_{2}}))$.  Let
$V=V_{1}\cup\ldots \cup V_{r}$ be the decomposition of $V$ into
irreducible components, which induces an isomorphism $R(V)\cong
R(V_{1})\times \ldots \times R(V_{r})$ (see e.g. \cite[III,
  Proposition~2.8]{kunz}). So the number of minimal primes of $I$ is
equal to the number of elements in a maximal complete set of
orthogonal idempotents of $R(V)\cong
Q(k[x_{2},x_{3}]/(x_{2}^{m_{3}}-x_{3}^{m_{2}}))$.
\end{proof}

\begin{Proposition}\label{cardd}
Let $I$ be an HN ideal and $m(I)\in\bn^{3}$ its associated integer
vector. If $a_{1}=1$ and ${\rm gcd}(m_{2}(I),m_{3}(I))=d$ is prime,
then ${\rm gcd}(m(I))=d$ as well and $I$ is not a prime
ideal. Moreover, ${\rm cardinal}\, {\rm Min}(A/I)\leq d$. Furthermore,
if ${\rm char}(k)=d$, then $I$ is primary, is not radical, and
$\fp_{m(I)}^{(d)}\subsetneq I\subsetneq \fp_{m(I)}$.
\end{Proposition}

\begin{proof}
Let us prove first that $I$ is not a prime ideal. Note that once this
is established, then by Theorem~\ref{theorem-prime}, ${\rm
  gcd}(m(I))\neq 1$; but since ${\rm gcd}(m_{2}(I),m_{3}(I))=d$ is
prime, ${\rm gcd}(m(I))$ must be equal to $d$ as well. In particular,
by Theorem~\ref{bound}, ${\rm cardinal}\, {\rm Min}(A/I)\leq
1+(d(d-1)/d)=d$.

Write $m(I)=m=(m_{1},m_{2},m_{3})\in\bn^{3}$ and $m_{2}=dn_{2}$,
$m_{3}=dn_{3}$ with ${\rm gcd}(n_{2},n_{3})=1$. Then
\begin{eqnarray*}
x_{2}^{m_{3}}-x_{3}^{m_{2}}=(x_{2}^{n_{3}}-x_{3}^{n_{2}})(x_{2}^{(d-1)n_{3}}+
x_{2}^{(d-2)n_{3}}x_{3}^{n_{2}}+\ldots
+x_{2}^{n_{3}}x_{3}^{(d-2)n_{2}}+x_{3}^{(d-1)n_{2}}),
\end{eqnarray*}
where $x_{2}^{n_{3}}-x_{3}^{n_{2}}$ is irreducible (by e.g.
\cite[Lemma~10.15]{fossum} or Lemma~\ref{mainlemma} above). 

Set $B=k[x_{2},x_{3}]$, $\fa=(x_{2}^{n_{3}}-x_{3}^{n_{2}})$ and
$\fb=(x_{2}^{(d-1)n_{3}}+\ldots +x_{3}^{(d-1)n_{2}})$, the
corresponding ideals generated in $B$, and $C=B/\fa$. In $C$,
$x_{2}^{n_{3}}=x_{3}^{n_{2}}$ so that in $C$,
$x_{2}^{(d-1)n_{3}}+\ldots
+x_{3}^{(d-1)n_{2}}=dx_{2}^{(d-1)n_{3}}$. Hence, if ${\rm char}(k)\neq
d$, $\fa+\fb$ contains the element $x_{2}^{(d-1)n_{3}}$ and so
$(B/(\fa +\fb))_{x_{2}}$ becomes the zero ring. In other words, $\fa
B_{x_{2}}$ and $\fb B_{x_{2}}$ are relatively prime ideals of
$B_{x_{2}}$. By the Chinese remainder theorem (see e.g \cite[II,
  Prop.~1.7]{kunz}),
\begin{eqnarray*}
B_{x_{2}}/(x_{2}^{m_{3}}-x_{3}^{m_{2}})B_{x_{2}}\cong
(B_{x_{2}}/\fa B_{x_{2}})\times (B_{x_{2}}/\fb B_{x_{2}}).
\end{eqnarray*}
In particular, by Proposition~\ref{rv}, 
\begin{eqnarray*}
(A/I)_{x_{2}x_{3}}\cong
  (k[x_{2},x_{3}]/(x_{2}^{m_{3}}-x_{3}^{m_{2}}))_{x_{2}x_{3}}\cong
  (B/\fa)_{x_{2}x_{3}}\times (B/\fb)_{x_{2}x_{3}},
\end{eqnarray*}
which is not a domain. In particular, $I$ is not prime.

If ${\rm char}(k)=d$, then
$x_{2}^{m_{3}}-x_{3}^{m_{2}}=(x_{2}^{n_{3}}-x_{3}^{n_{2}})^{d}$.
Keeping the same notations as above, by Proposition~\ref{rv} again,
\begin{eqnarray*}
(A/I)_{x_{2}x_{3}}\cong
  (k[x_{2},x_{3}]/(x_{2}^{m_{3}}-x_{3}^{m_{2}}))_{x_{2}x_{3}}\cong
  (B/\fa^{d})_{x_{2}x_{3}},
\end{eqnarray*}
where $\fa$ is a prime ideal in $B$ and a complete intersection (in
fact principal), so $\fa^{d}=\fa^{(d)}$, the $d$-th symbolic power,
and $\fa^{d}$ is $\fa$-primary. Hence the nilradical of $B/\fa^{d}$ is
the (non-zero) prime ideal $\fa/\fa^{d}$, whose $d$-th power is zero;
in fact the colength of $\fa^{d}$ at $\fa$ is precisely $d$. Since
$x_{2}x_{3}$ lies outside $\fa$, this structure is preserved when we
localize at the element $x_{2}x_{3}$. In the light of the isomorphism
established above, we deduce that $(A/I)_{x_{2}x_{3}}$ has a
(non-zero) prime nilradical with $d$-th power equal to zero, so this
prime radical must therefore be $(\fp_{m}/I)_{x_{2}x_{3}}$. Since $I$
is unmixed and $x_{2}x_{3}$ is regular modulo $I$, we must have that
$I$ is $\fp_{m}$-primary and $\fp_{m}^{d}\subseteq I\subsetneq
\fp_{m}$. In particular, $\fp_{m}^{(d)}\subseteq I$ . Furthermore,
$\fp_{m}^{(d)}$ equals $I$ if and only if they have the same local
co-length at $\fp_{m}$. Now $A_{\mathfrak{p}_{m}}$ is a regular local
ring of dimension $2$, so the local colength of $A/\fp_{m}^{(d)}$ at
$\fp_{m}$ is $d(d+1)/2$. Hence the co-lengths agree if and only if
$d(d+1)/2=d$, i.e., $d =1$. So for $d>1$, $I$ properly contains
$\fp_{m}^{(d)}$.
\end{proof}

For the concrete case $d=2$, we obtain the following result.

\begin{Corollary}\label{2primes}
Let $I$ be an HN ideal and $m(I)\in\bn^{3}$ its associated integer
vector. Suppose that $a_{1}=1$ and ${\rm
  gcd}(m_{2}(I),m_{3}(I))=2$. In this case, if ${\rm char}(k)\neq 2$,
then $I$ is radical and equal to the intersection of exactly two prime
ideals; on the other hand, if ${\rm char}(k)=2$, then $I$ is primary,
is not radical, and $\fp_{m(I)}^{(2)}\subsetneq I\subsetneq
\fp_{m(I)}$.
\end{Corollary}

\begin{proof}
By Proposition~\ref{cardd}, we have only to show that $I$ is radical
whenever ${\rm char}(k)\neq 2$. Thus suppose that ${\rm char}(k)\neq
2$. Write $m(I)=(m_{1},m_{2},m_{3})\in\bn^{3}$ and $m_{2}=2n_{2}$,
$m_{3}=2n_{3}$ with ${\rm gcd}(n_{2},n_{3})=1$. Then
$x_{2}^{m_{3}}-x_{3}^{m_{2}}=(x_{2}^{n_{3}}-x_{3}^{n_{2}})
(x_{2}^{n_{3}}+x_{3}^{n_{2}})$ is a decomposition into prime factors
in $B=k[x_{2},x_{3}]$ (see e.g. \cite[Lemma~10.15]{fossum} or
Lemma~\ref{mainlemma} above). Let $\fa_{1}=
(x_{2}^{n_{3}}-x_{3}^{n_{2}})$ and
$\fa_{2}=(x_{2}^{n_{3}}+x_{3}^{n_{2}})$ be the corresponding prime
ideals generated in $B$. Then $\fa_{1}\neq \fa_{2}$ and, localising at
$x_{2}$, $\fa_{1}B_{x_{2}}$ and $\fa_{2}B_{x_{2}}$ are two relatively
prime ideals of $B_{x_{2}}$. By the Chinese remainder theorem (see e.g
\cite[II, Prop.~1.7]{kunz}),
\begin{eqnarray*}
B_{x_{2}}/(x_{2}^{m_{3}}-x_{3}^{m_{2}})B_{x_{2}}\cong
(B_{x_{2}}/(x_{2}^{n_{3}}-x_{3}^{n_{2}})B_{x_{2}})\times
(B_{x_{2}}/(x_{2}^{n_{3}}+x_{3}^{n_{2}})B_{x_{2}}).
\end{eqnarray*}
In particular, by Proposition~\ref{rv}, 
\begin{eqnarray*}
(A/I)_{x_{2}x_{3}}\cong
  (k[x_{2},x_{3}]/(x_{2}^{m_{3}}-x_{3}^{m_{2}}))_{x_{2}x_{3}}\cong
(B/\fa_{1})_{x_{2}x_{3}}\times (B/\fa_{2})_{x_{2}x_{3}},
\end{eqnarray*}
which is a reduced ring. Since $x_{2}x_{3}$ is regular modulo $I$,
$A/I$ is reduced and $I$ is radical.
\end{proof}

\begin{Example}\label{again}{\rm
Consider (again) the HN ideal $I$ of Example~\ref{char2}, which
satisfies the hypotheses of Corollary~\ref{2primes}, i.e., $a_{1}=1$
and ${\rm gcd}(m_{2}(I),m_{3}(I))=2$. Thus one can conclude that if
${\rm char}(k)\neq 2$, $I$ is radical and equal to the intersection of
two primes, whereas if ${\rm char}(k)=2$, then $I$ is primary, is not
radical, and $\fp_{m(I)}^{(2)}\subsetneq I\subsetneq
\fp_{m(I)}$. }\end{Example}

For the concrete case $d=3$, we have the following result.

\begin{Corollary}\label{3primes}
Let $I$ be an HN ideal and $m(I)\in\bn^{3}$ its associated integer
vector. Suppose that $a_{1}=1$ and ${\rm
  gcd}(m_{2}(I),m_{3}(I))=3$. In this case, if ${\rm char}(k)\neq 2,3$
and $k$ contains a square root of $-3$, then $I$ is radical and equal
to the intersection of exactly three prime ideals; if ${\rm
  char}(k)\neq 3$ and either ${\rm char}(k)=2$ or else $k$ does not
contain a square root of $-3$, then $I$ is radical and equal to the
intersection of exactly two prime ideals; finally, if ${\rm
  char}(k)=3$, then $I$ is primary, is not radical, and
$\fp_{m(I)}^{(3)}\subsetneq I\subsetneq \fp_{m(I)}$.
\end{Corollary}

\begin{proof}
Write $m(I)=(m_{1},m_{2},m_{3})\in\bn^{3}$ and $m_{2}=3n_{2}$,
$m_{3}=3n_{3}$ with ${\rm gcd}(n_{2},n_{3})=1$. Then
$x_{2}^{m_{3}}-x_{3}^{m_{2}}$ decomposes as
$(x_{2}^{n_{3}}-x_{3}^{n_{2}})
(x_{2}^{2n_{3}}+x_{2}^{n_{3}}x_{3}^{n_{2}}+x_{3}^{2n_{2}})$, where
$x_{2}^{n_{3}}-x_{3}^{n_{2}}$ is irreducible (by e.g.
\cite[Lemma~10.15]{fossum} or Lemma~\ref{mainlemma} above). On the
other hand, any proper factor of
$x_{2}^{2n_{3}}+x_{2}^{n_{3}}x_{3}^{n_{2}}+x_{3}^{2n_{2}}$ is a proper
factor of $x_{2}^{m_{3}}-x_{3}^{m_{2}}$. Hence, by
Lemma~\ref{mainlemma} again, any decomposition of
$x_{2}^{2n_{3}}+x_{2}^{n_{3}}x_{3}^{n_{2}}+x_{3}^{2n_{2}}$ must be of
the form $(x_{2}^{n_{3}}+\lambda
x_{3}^{n_{2}})(x_{2}^{n_{3}}+\lambda^{-1} x_{3}^{n_{2}})$, $\lambda\in
k\setminus \{0\}$. Therefore $\lambda +\lambda^{-1}=1$ and hence
$\lambda^{2}-\lambda +1=0$.

Suppose now that ${\rm char}(k)\neq 2,3$ and that $k$ contains a
square root of $-3$. Take $\lambda_{0}\in k$ a solution of
$\lambda^{2}-\lambda +1=0$. Set $B=k[x_{2},x_{3}]$ and
$\fa_{1}=(x_{2}^{n_{3}}-x_{3}^{n_{2}})$,
$\fa_{2}=(x_{2}^{n_{3}}+\lambda_{0} x_{3}^{n_{2}})$ and
$\fa_{3}=(x_{2}^{n_{3}}+\lambda_{0}^{-1}x_{3}^{n_{2}})$ the distinct
prime ideals generated in $B$. Localising at $x_{2}$,
$\fa_{1}B_{x_{2}}$, $\fa_{2}B_{x_{2}}$ and $\fa_{3}B_{x_{2}}$ become
three pairwise relatively prime ideals of $B_{x_{2}}$. By the Chinese
remainder theorem (see e.g \cite[II, Prop.~1.7]{kunz}),
\begin{eqnarray*}
B_{x_{2}}/(x_{2}^{m_{3}}-x_{3}^{m_{2}})B_{x_{2}}\cong
(B_{x_{2}}/\fa_{1}B_{x_{2}})\times (B_{x_{2}}/\fa_{2}B_{x_{2}})\times
(B_{x_{2}}/\fa_{2}B_{x_{3}}).
\end{eqnarray*}
In particular, by Proposition~\ref{rv}, 
\begin{eqnarray*}
(A/I)_{x_{2}x_{3}}\cong
  (k[x_{2},x_{3}]/(x_{2}^{m_{3}}-x_{3}^{m_{2}}))_{x_{2}x_{3}}\cong
  (B/\fa_{1})_{x_{2}x_{3}}\times (B/\fa_{2})_{x_{2}x_{3}}\times
  (B/\fa_{3})_{x_{2}x_{3}},
\end{eqnarray*}
which is a reduced ring. By \cite[III, Proposition~4.23]{kunz},
$Q(A/I)$ is the product of three fields. So $I$ has exactly three
minimal primes. Moreover, since $(A/I)_{x_{2}x_{3}}$ is reduced and
$x_{2}x_{3}$ is regular modulo $I$, $A/I$ is reduced and $I$ is
radical. Thus $I$ is the intersection of exactly three prime ideals.

If ${\rm char}(k)\neq 3$ and either ${\rm char}(k)=2$ or else $k$ does
not contain a square root of $-3$, then
$\fa=(x_{2}^{n_{3}}-x_{3}^{n_{2}})$ and
$\fb=(x_{2}^{2n_{3}}+x_{2}^{n_{3}}x_{3}^{n_{2}}+x_{3}^{2n_{2}})$ are
two distinct prime ideals of $B=k[x_{2},x_{3}]$. Remark that
$3x_{2}^{2n_{3}}=(2x_{2}^{n_{3}}+x_{3}^{n_{2}})(x_{2}^{n_{3}}-x_{3}^{n_{2}})+
(x_{2}^{2n_{3}}+x_{2}^{n_{3}}x_{3}^{n_{2}}+x_{3}^{2n_{2}})$. Thus,
localizing at $x_{2}$, $\fa B_{x_{2}}$ and $\fb B_{x_{2}}$ become
two relatively prime ideals of $B_{x_{2}}$. Proceeding as before, one
deduces that $(A/I)_{x_{2}x_{3}}$ is reduced and that $Q(A/I)$ is the
product of two fields. So $I$ is radical and equal to the intersection
of exactly two prime ideals.

If ${\rm char}(k)=3$, then finish by applying Proposition~\ref{cardd}.
\end{proof}

\begin{Example}\label{again?}{\rm 
Let $I=(x_{1}^{4}-x_{2}^{2}x_{3}^{3},x_{2}^{5}-x_{1}x_{3}^{3},
x_{3}^{6}-x_{1}^{3}x_{2}^{3})$ be the HN ideal associated to
\begin{eqnarray*}
\cm=\left(\begin{array}{ccc}x_{1}^{}&x_{2}^{3}&x_{3}^{3}
  \\ x_{2}^{2}&x_{3}^{3}&x_{1}^{3}\end{array}\right).
\end{eqnarray*}
Here $a_{1}=1$ and $m(I)=m=(21,15,18)$, so ${\rm
  gcd}(m_{2},m_{3})=3$. Thus one can apply
Corollary~\ref{3primes}. For instance, if $k=\bc$, $I$ is radical and
equal to the intersection of three prime ideals, whereas if $k=\bq$,
$I$ is radical and equal to the intersection of two prime ideals. On
the other hand, if ${\rm char}(k)=3$, $I$ is primary, is not radical,
and $\fp_{m}^{(3)}\subsetneq I\subsetneq \fp_{m}$.

In any case, the Herzog ideal $\fp_{m}$ associated to $m=(21,15,18)$
is a minimal prime of $I$. A simple computation shows that
$\fp_{m}=(x_{1}^{3}-x_{2}^{3}x_{3}^{},x_{2}^{4}-x_{1}^{2}x_{3}^{},
x_{3}^{2}-x_{1}^{}x_{2}^{})$.

If $k=\bq$, {\sc Singular} (see \cite{singular}) gives for the minimal
prime of $I$ other than $\fp_{m}$ the ideal
\begin{eqnarray*}
(x_{1}^{4}-x_{2}^{2}x_{3}^{3},x_{2}^{5}-x_{1}^{}x_{3}^{3},
  x_{1}^{3}x_{3}+x_{1}x_{2}^{4}+x_{2}^{3}x_{3}^{2},
  x_{1}^{2}x_{2}^{2}+x_{1}x_{2}x_{3}^{2}+x_{3}^{4},
  x_{1}^{3}x_{2}+x_{1}^{2}x_{3}^{2}+x_{2}^{4}x_{3}),
\end{eqnarray*}
which is not a binomial ideal. We recall that, according to the work
of Eisenbud and Sturmfels \cite{es}, if $k=\bc$, all the associated
primes of $I$ must be binomial.

In this connection, we remark that the results of \cite{es}, and the
rich combinatorial and algorithmic theory of binomial ideals that grew
from it, could well throw light on the questions left open in this
paper. We intend to pursue this line of enquiry in future work.
}\end{Example}

\noindent {\sc Acknowledgement}. We gratefully acknowledge the support
of the Glasgow Mathematical Journal Trust Fund and the MEC Grant
MTM2007-67493 in carrying out this research, and should like to thank
the School of Mathematics, University of Edinburgh, and the
Departament de Matem\`atica Aplicada 1, Universitat Polit\`ecnica de
Catalunya, Barcelona, for their hospitality. We would like to thank
the referee for a careful reading of the paper that has helped to
improve its presentation.

\begin{center}
\parbox[t]{7cm}{\footnotesize

\noindent Liam O'Carroll

\noindent Maxwell Institute for Mathematical Sciences

\noindent School of Mathematics

\noindent University of Edinburgh

\noindent EH9 3JZ, Edinburgh, Scotland

\noindent L.O'Carroll@ed.ac.uk } 
\,
\parbox[t]{6.2cm}{\footnotesize

\noindent Francesc Planas-Vilanova

\noindent Departament de Matem\`atica Aplicada~1

\noindent Universitat Polit\`ecnica de Catalunya

\noindent Diagonal 647, ETSEIB

\noindent 08028 Barcelona, Catalunya

\noindent francesc.planas@upc.edu}
\end{center}
\end{document}